\documentclass[oneside,english]{amsart}
\usepackage[T1]{fontenc}
\usepackage[a4paper]{geometry}
\usepackage{amstext}
\usepackage{amsthm}
\usepackage{amssymb}

\makeatletter
\numberwithin{equation}{section}
\numberwithin{figure}{section}
\theoremstyle{plain}
\newtheorem{thm}{\protect\theoremname}[section]
\theoremstyle{plain}
\newtheorem{assumption}[thm]{\protect\assumptionname}
\theoremstyle{definition}
\newtheorem{example}[thm]{\protect\examplename}
\theoremstyle{remark}
\newtheorem{rem}[thm]{\protect\remarkname}
\theoremstyle{plain}
\newtheorem{lem}[thm]{\protect\lemmaname}
\theoremstyle{definition}
\newtheorem{defn}[thm]{\protect\definitionname}
\theoremstyle{plain}
\newtheorem{prop}[thm]{\protect\propositionname}
\theoremstyle{plain}
\newtheorem{cor}[thm]{\protect\corollaryname}

\renewcommand{\div}{\operatorname{div}}

\newcommand{\sgn}{\operatorname{sgn}}

\newcommand{\eps}{\varepsilon}

\usepackage{cite}
\usepackage{hyperref}

\numberwithin{equation}{section}

\newcommand{\Acal} {{\mathcal A}}
\newcommand{\Bcal} {{\mathcal B}}

\newcommand{\Ecal} {{\mathcal E}}
\newcommand{\Fcal} {{\mathcal F}}

\newcommand{\Lcal} {{\mathcal L}}

\newcommand{\Z}{\mathbb{Z}}
\newcommand{\R}{\mathbb{R}}
\newcommand{\N}{\mathbb{N}}

\renewcommand{\P}{\mathbb{P}}
\newcommand{\E}{\mathbb{E}}

\newcommand{\C}{\mathbb{C}}

\newcommand{\A}{\mathbb{A}}
\newcommand{\T}{\mathbb{T}}

\newcommand{\brrr}[1]{\sideset{_{S^\ast}}{_{S}}{\mathop{\left\langle #1 \right\rangle}}}


\AtBeginDocument{
  
}

\makeatother

\usepackage{babel}
\providecommand{\assumptionname}{Assumption}
\providecommand{\corollaryname}{Corollary}
\providecommand{\definitionname}{Definition}
\providecommand{\examplename}{Example}
\providecommand{\lemmaname}{Lemma}
\providecommand{\propositionname}{Proposition}
\providecommand{\remarkname}{Remark}
\providecommand{\theoremname}{Theorem}

\begin{document}
\author{Jonas M. T\"{o}lle}
\address{Universit\"{a}t Augsburg\\
Institut f\"{u}r Mathematik\\
86135 Augsburg\\
Germany}
\email{jonas.toelle@math.uni-augsburg.de}
\thanks{The author would like to thank Max von Renesse, Hiroshi Kawabi and
Lisa Beck for stimulating discussions on the topic of this work. The
author is indebted to the anonymous referees for their suggestions
of several improvements. Continuous support by the M.O.P.S. program
is gratefully acknowledged.}
\date{\today}
\title[Stochastic evolution equations with singular drift and gradient noise]{Stochastic evolution equations with singular drift and gradient noise
via curvature and commutation conditions}
\begin{abstract}
We prove existence and uniqueness of solutions to a nonlinear stochastic
evolution equation on the $d$-dimensional torus with singular $p$-Laplace-type
or total variation flow-type drift with general sublinear doubling
nonlinearities and Gaussian gradient Stratonovich noise with divergence-free
coefficients. Assuming a weak defective commutator bound and a curvature-dimension
condition, the well-posedness result is obtained in a stochastic variational
inequality setup by using resolvent and Dirichlet form methods and
an approximative It\^{o}-formula.
\end{abstract}

\maketitle

\section{Introduction}

We shall study the nonlinear stochastic partial differential equation
(nonlinear SPDE) with singular drift and gradient-type multiplicative
Stratonovich noise 
\begin{equation}
\begin{aligned}dX_{t} & =\div(a^{\ast}\phi(a\nabla X_{t}))\,dt+\langle b\nabla X_{t},\circ\,dW_{t}\rangle,\quad t\in(0,T],\\
X_{0} & =x.
\end{aligned}
\label{eq:formal-SPDE}
\end{equation}
In order to eliminate boundary or curvature effects from the underlying
space, we shall consider the SPDE on $\T^{d}=\R^{d}/\Z^{d}$, $d\ge1$.
Here, $a:\T^{d}\to\R^{d\times d}$, $b:\T^{d}\to\R^{N\times d}$ are
$C^{1}$-coefficient fields (where $a^{\ast}$ denotes the transpose
of $a$) and $\{W_{t}\}_{t\ge0}$ is an $\R^{N}$-valued standard
Wiener process, $N\ge1$. In this work, we shall prove well-posedness
for initial data $x\in L^{2}(\T^{d})$ or, more generally, for $x\in L^{2}(\Omega,\Fcal_{0},\P;L^{2}(\T^{d}))$.
We assume a linear growth condition on $\phi:\R^{d}\to\R^{d}$, where
we also may consider the multi-valued case $\phi:\R^{d}\to2^{\R^{d}}$.
Furthermore, a geometric curvature-dimension condition on the Riemannian
metric $g_{a}:=(a^{\ast}a){}^{-1}$ and a (defective) commutator estimate
for the first order operator $u\mapsto b\nabla u$ are assumed.

The drift operator $u\mapsto\div(a^{\ast}\phi(a\nabla u))$ is a distorted
$p$-Laplace-type operator, when $\phi(\zeta)=|\zeta|^{p-2}\zeta$,
$p\in[1,2]$, which reduces to the linear diffusion operator $u\mapsto\div(a^{\ast}a\nabla u)$
for $p=2$ and includes multi-valued examples as the total variation
flow operator $u\mapsto\div(\sgn(\nabla u))$ for $p=1$, $a=1$,
which is also called $1$-Laplace. The equation is perturbed by independent
Brownian motions $\{W_{t}^{i}\}_{t\ge0}$, $1\le i\le N$, driven
by family of divergence-free $C^{1}$-vector fields $b_{i}$, $1\le i\le N$
acting in gradient direction, which may be degenerate, so that the
``deterministic'' PDE case is covered for $b\equiv0$.

Equations of the type \eqref{eq:formal-SPDE}, have previously been
studied in \cite{BBHT13} (for the case $p>1$), in \cite{Ciotir:2016fe}
(for the case of a domain with symmetries and Neumann boundary conditions),
in \cite{MunteanuRoeckner2016} (for Dirichlet boundary conditions),
and in \cite{Barbu:2017dr} (by an approach with weak solutions in
the sense of distributions). As the coefficient field of the noise
term is given by an unbounded operator in space, the It\^{o}-analogue
of \eqref{eq:formal-SPDE} is ill-posed in general. The equation is
discussed in \cite[Section 7]{BR15} as an example for an approach
via a transformation by a group of random multipliers. Note that our
approach does not rely on a transformation of \eqref{eq:formal-SPDE}
to a random PDE --- rather than that, we obtain the unique solutions
in terms of (stochastic) variational inequalities by a multi-step
approximation procedure and a perturbation argument which relies on
a weak defective commutator bound formulated in terms of Dirichlet
forms, see \cite{FOT,MR} for this notion. The conditions are discussed
in Section \ref{sec:Ingredients} below. In particular, we need a
curvature bound related to heat kernel estimates and lower bounds
for Ricci curvature in order to derive a priori estimates for the
singular equation. The higher order a priori estimates, which we need
for the approximation procedure in the proof of the main result, are
then derived from the defective commutator bounds. In \cite{Toelle2018proceedings},
some previous results and the rough idea, which this work is based
on, have been proposed by the author. In the future, equations involving
more general Dirichlet operators, e.g. of nonlocal type, could be
considered.

As the drift term in the SPDE \eqref{eq:formal-SPDE} is singular,
it is a known issue that solutions to the SPDE do not satisfy an It\^{o}-equation
in general (on these lines, see e.g. \cite{GessToelle14}), when,
for instance, there is no Sobolev embedding for the energy of the
drift available (as is e.g. in \cite{Liu09} for $p>1\vee\frac{2d}{d+2}$).
See \cite{Marinelli:2018ej} for an approach for nonlinearities that
satisfy an superlinearity condition at infinity and see \cite{ValletZimmermann2019}
for the case that the drift is perturbed with a first-order transport
term. In our situation, the It\^{o}-Stratonovich correction and the
special linear structure of the noise adds some regularity to the
equation, however, even for initial data in $H^{1}(\mathbb{T}^{d})$
and for $p\approx1$, to the best of our knowledge, the (limit) solutions
to \eqref{eq:formal-SPDE} can merely be characterized in terms of
stochastic variational inequalities (SVI), see e.g. \cite{VBMR}.
Among others, SVI-solutions to stochastic evolution equations have
also been discussed in \cite{Rascanu:2017dp,BenRas,Gess:2016kda,Gess:2015gw,GT:2016he}.
As seen in \cite{GessToelle15}, the SVI-approach is quite robust
under perturbations of the convex-subpotential-type drift with respect
to the Mosco-topology. In this work, we prove existence and uniqueness
for initial data in $L^{2}(\T^{d})$, see our main result Theorem
\ref{thm:mainmainthm} in Section \ref{sec:The-main-result} below.

Let us point out that we generalize the previously known results on
solutions to \eqref{eq:formal-SPDE} in the following aspects. As
we assume periodic boundary conditions, we are able to dispense with
the requirement of the boundary to be assumed to be $C^{3}$ and coefficient
fields to be assumed to be perpendicular to the boundary and $C^{2}$.
We do merely require that the driving vector fields $b_{i}$, $1\le i\le N$
are $C^{1}$, divergence-free and pointwise linearly independent ---
we do not need to assume commutation\footnote{We remark that there is another way to dispense with commutation via
a transformation to a flow of diffeomorphisms, see e.g. \cite{BrzeZniak:1988gf}.} here, as is done e.g. in the classical linear SPDE case in \cite{DaPrato:1982wn,DaPrato:1982ci}.
The introduction of the deformed diffusivity is new and may possibly
be applied to more general compact Riemannian manifolds than the torus
in the future. Apart from the geometric deformation, we also include
more general nonlinearities (in the spirit of \cite[Section 7]{GessToelle14}),
extending the results of \cite{Ciotir:2016fe,GessToelle15} from homogeneous
nonlinearities to ones satisfying a doubling condition, see Example
\ref{exa:Condition-(N)} below. In fact, the nonlinearities might
become multi-valued so that equation \eqref{eq:formal-SPDE} and its
It\^{o}-Stratonovich corrected form \eqref{eq:main-spde} below are
actually given by the more general subpotential equation \eqref{eq:main-spde-1}
below, which involves the relaxed convex potentials. In the sequel,
all of our analysis is generally referring to equation \eqref{eq:main-spde-1}
below rather than to the formal equation \eqref{eq:formal-SPDE}. 

Possible future topics for equations of this type include ergodicity
and uniqueness of invariant measures for Markov semigroups (compare
also \cite{GT:2016he,LiuToe1,BDP06,ESvR,ESvRS} for additive noise),
stability under perturbations of the drift or noise (cf. \cite{GessToelle15,CiotToe})
and regularity (see e.g. \cite{Breit:2016dga,Breit:2014cda}) or an
approach via entropy solutions \cite{DareiotisGerencserGess2019}
or renormalized solutions \cite{SapountzoglouZimmermann2019}. We
also refer to \cite{DareiotisGess2018,Turra2018,FehrmanGess2019}
for porous media type stochastic equations with gradient noise.

\subsection*{Notation}

Denote by $\mathbb{T}^{d}:=\mathbb{R}^{d}\big/\mathbb{Z}^{d}$ the
standard flat torus of dimension $d\ge1$, equipped with the $d$-dimensional
Lebesgue measure $d\xi$ on its Borel $\sigma$-algebra $\mathcal{B}(\mathbb{T}^{d})$.
On $\R^{d}$, the Euclidean norm and inner product  are denoted by
$|\cdot|$ and $\langle\cdot,\cdot\rangle$ respectively. For $\alpha\in\R$,
we denote the linear operator $u\mapsto\alpha u$ (on some vector
space) simply by $\alpha$. Denote $H:=L^{2}(\mathbb{{T}}^{d})$,
$S:=H^{1}(\mathbb{{T}}^{d})$. Let $S^{\ast}$ denote the topological
dual of $S$. Note that the embedding $S\hookrightarrow H$ is dense
and compact. Let $(\mathcal{E},D(\mathcal{E}))$ be the \emph{Dirichlet
form} of the \emph{Laplace-Beltrami operator} $L:=\Delta:=\div(\nabla\cdot)$
on $H$, that is, $D(\mathcal{{E}}):=S$ and
\[
\mathcal{E}(u,v):=\int_{\mathbb{T}^{d}}\langle\nabla u,\nabla v\rangle\,d\xi,\quad u,v\in D(\mathcal{E}),
\]
see \cite{FOT,MR}. Let $(G_{\alpha})_{\alpha>0}$ be the \emph{resolvent}
of $L$, i.e., for $\alpha>0$, $G_{\alpha}u:=(\alpha-L)^{-1}u$,
$u\in H$. For convenience, we shall also introduce the alternative
resolvent $J_{\delta}u:=(1-\delta L)^{-1}u$, where $u\in H$ and
$\delta>0$. Clearly, $J_{\delta}=\frac{1}{\delta}G_{1/\delta}$ for
every $\delta>0$. The resolvent $J_{\delta}$, when considered both
as a map from $H$ to $H$ or as a map from $S$ to $S$, is a contraction.
Denote the \emph{Yosida-approximation} of $L$\emph{ }by $L^{(\delta)}u:=LJ_{\delta}u=\frac{{1}}{\delta}(J_{\delta}-1)u$,
$u\in H$. $L^{(\delta)}:H\to H$ is a negative definite bounded linear
operator with operator norm bounded by $\frac{1}{\delta}$. Furthermore,
for $\alpha>0$, $u,v\in D(\mathcal{E})$, let $\mathcal{E}_{\alpha}(u,v):=\mathcal{E}(u,v)+\alpha(u,v)_{H}$
be inner products for $S$ with the property that $\Ecal_{\alpha_{1}}$,
$\Ecal_{\alpha_{2}}$ are mutually equivalent for $\alpha_{1},\alpha_{2}>0$.
For $\beta>0$, define also \emph{approximate forms} $\mathcal{E}^{(\beta)}(u,v):=\beta(u-\text{\ensuremath{\beta}}G_{\text{\ensuremath{\beta}}}u,v)_{H}$,
$u,v\in H$, see e.g. \cite[Chapter I, p. 20]{MR}. Set also $\mathcal{E}_{\alpha}^{(\beta)}(u,v):=\mathcal{E}^{(\beta)}(u,v)+\alpha(u,v)_{H}$,
$\alpha,\beta>0$, $u,v\in H$. Denote by $(P_{t})_{t\ge0}$ the \emph{heat
semigroup }associated to $L$. We shall denote the \emph{trace} of
a square matrix by $\operatorname{Tr}$.

\subsection*{Organization of the paper}

In Section \ref{sec:Ingredients}, which is subdivided into several
subsections, we shall discuss the different terms of equation \eqref{eq:formal-SPDE}
and the main assumptions on those. Here, also the main hypotheses
of this work are stated and discussed. In Section \ref{sec:Stochastic-VI},
we shall introduce and discuss the concept of so-called \emph{stochastic
variational inequality (SVI) solutions} to our equation. Section \ref{sec:The-main-result},
which is subdivided into several subsections, contains the statement
and the proof of our main result, including the a priori estimates,
the passage to the limit of the approximating equations, the existence
and the uniqueness result. In Appendix \ref{sec:app-ex}, we recall
the conditions from \cite{GessToelle14} which are needed to guarantee
the existence of approximating solutions to equation \eqref{eq:formal-SPDE}.
The concluding Appendix \ref{sec:Proof} contains the postponed proof
of Proposition \ref{prop:BEcond}.

\section{\label{sec:Ingredients} Main hypotheses}

Let $\{W_{t}\}_{t\ge0}$ be a \emph{Wiener process} on $\R^{N}$,
modeled on a filtered probability space $(\Omega,\Fcal,\{\Fcal_{t}\}_{t\ge0},\P)$
that satisfies the usual conditions.

Consider the following It\^{o}-SPDE in $H$,
\begin{equation}
\begin{aligned}dX_{t} & =\div(a^{\ast}\phi(a\nabla X_{t}))\,dt+\frac{1}{2}L^{b}X_{t}\,dt+\langle b\nabla X_{t},dW_{t}\rangle,\quad t\in(0,T],\\
X_{0} & =x.
\end{aligned}
\label{eq:main-spde}
\end{equation}

Equation \eqref{eq:main-spde} is the formal analogue to equation
\eqref{eq:formal-SPDE}, after passing over to an (merely formal)
It\^{o}-Stratonovich correction, see e.g. \cite{Kurtz:1995vc,kunita1997stochastic}
and compare with the previous works \cite{Barbu:2017dr,Ciotir:2016fe,BBHT13,Toelle2018proceedings}.

The precise assumptions on $a,b$ and $\phi$ will be stated below.
The operator $\frac{1}{2}L^{b}$ is a realization of $u\mapsto\frac{1}{2}\div[b^{\ast}b\nabla u]$,
which equals the It\^{o}-Stratonovich correction $u\mapsto\frac{1}{2}\operatorname{Tr}[b\nabla(b\nabla u)]$
of the noise term $u\mapsto\langle b\nabla u,\circ dW_{t}\rangle$,
whenever $\div b_{i}=0$ for all $1\le i\le N$.

\subsection{The nonlinearity $\phi$}

To improve precision once more, instead of \eqref{eq:main-spde},
we shall in fact study the following It\^{o}-SPDE (stochastic evolution
equation / inclusion\footnote{The set membership symbol is to be understood as an equality whenever
the r.h.s. is single-valued.})
\begin{equation}
\begin{aligned}dX_{t} & \in-\partial\Psi(X_{t})\,dt+\frac{1}{2}L^{b}X_{t}\,dt+\langle b\nabla X_{t},dW_{t}\rangle,\quad t\in(0,T],\\
X_{0} & =x.
\end{aligned}
\label{eq:main-spde-1}
\end{equation}
where $\partial\Psi$ denotes the subdifferential\footnote{The \emph{subdifferential} $\partial F:H\to2^{H}$ of a convex, lower
semi-continuous, proper map $F:H\to[0,+\infty]$ is defined by $y\in\partial F(x)$,
$x,y\in H$, whenever $(y,z-x)_{H}\le F(z)-F(x)$ for every $z\in H$,
see e.g. \cite{Barbu:2012tq}.} of $\Psi$, which, in turn, is defined to be the lower semi-continuous
(l.s.c.) envelope of the map
\begin{equation}
\tilde{\Psi}:u\mapsto\begin{cases}
\int_{\mathbb{T}^{d}}\psi(a\nabla u)\,d\xi, & u\in H^{1}(\mathbb{{T}}^{d}),\\
+\infty, & u\in L^{2}(\mathbb{{T}}^{d})\setminus H^{1}(\mathbb{{T}}^{d}),
\end{cases}\label{eq:Psi-tilde}
\end{equation}
where $\psi:\mathbb{R}^{d}\to[0,\infty)$ is some \emph{convex potential}
of $\phi$, that is, $\langle\eta,\zeta-\xi\rangle\le\psi(\zeta)-\psi(\xi)$
for every $\eta\in\phi(\xi)$ and all $\xi,\zeta\in\mathbb{R}^{d}$.
The l.s.c. envelope is defined by
\[
\Psi(u):=\operatorname{cl}\tilde{{\Psi}}(u):=\inf\left\{ \liminf_{n\to\infty}~\tilde{{\Psi}}(u_{n})\;\vert~u_{n}\rightarrow u\in L^{2}(\T^{d})\;\text{{strongly}}\right\} ,
\]
see e.g. \cite{ABM} for further details.
\begin{assumption}
\label{assu:N} Let us assume that there exist constants $C,K>0$
and a map $\theta:[0,\infty)\to[0,\infty)$ such that
\begin{enumerate}
\item[(N)]  $\psi(\zeta)=\theta(|\zeta|)$, for all $\zeta\in\mathbb{R}^{d}$
and $\theta$ is convex, continuous, and satisfies $\theta(0)=0$,
$\lim_{r\to\infty}\theta(r)=\infty$, $\theta(r)\le C(1+|r|^{2})$,
$\theta(2r)\le K\theta(r)$ for $r\ge0$.
\end{enumerate}
\end{assumption}

Note that the doubling condition $\theta(2r)\le K\theta(r)$, $r\ge0$
is also known as $\Delta_{2}$-condition in the literature, see e.g.
\cite{RR2}. Our assumptions on $a$ shall be made precise further
below.
\begin{example}
\label{exa:Condition-(N)} Condition (N) is e.g. satisfied for the
following choices of $\psi$. For $\zeta\in\mathbb{R}^{d}$, let
\begin{description}
\item [{$p$-Laplace,~total~variation~flow}] $\psi(\zeta):=\frac{1}{p}|\zeta|^{p}$,
$p\in[1,2]$,
\item [{Logarithmic~diffusion}] $\psi(\zeta):=(1+|\zeta|)\log(1+|\zeta|)-|\zeta|$,
\item [{Minimal~surface~flow}] $\psi(\zeta):=\sqrt{1+|\zeta|^{2}}$,
\item [{Curve~shortening~flow}] $\psi(\zeta):=|\zeta|\arctan(|\zeta|)-\frac{1}{2}\log(|\zeta|^{2}+1)$.
\end{description}
\end{example}

\begin{rem}
Note that e.g. for $\psi(\zeta)=|\zeta|$, the subdifferential $\phi=\partial\psi$
becomes multi-valued, and the l.s.c. envelope $\Psi$ is a Radon measure
on the space of functions of bounded variation, see \cite{ABM} for
details. At this point it is enough to recall that the above l.s.c.
envelope exists in $L^{2}(\T^{d})$.
\end{rem}

We continue with two technical lemmas needed later.
\begin{lem}
\label{lem:Delta2} Assume that (N) holds. Then
\[
\langle\eta,\zeta\rangle\le K\psi(\zeta)\quad\forall\eta\in\phi(\zeta)\;\forall\zeta\in\mathbb{R}^{d}.
\]
\end{lem}

\begin{proof}
By (N) and \cite[Chapter II, Example 8.A]{Show}, there exists a function
$\theta_{+}^{\prime}:[0,\infty)\to[0,\infty)$ such that $\theta_{+}^{\prime}$
is non-decreasing, right-continuous and satisfies
\[
\theta(s)=\int_{0}^{s}\theta_{+}^{\prime}(r)\,dr,\quad s\ge0.
\]
Also, for $r\ge0$, $\theta_{+}^{\prime}(r)\ge\sup_{s\in\partial\theta(r)}|s|$.
Using (N) again, for $s\ge0$,
\[
K\theta(s)\ge\theta(2s)=\int_{0}^{2s}\theta_{+}^{\prime}(r)\,dr\ge\int_{s}^{2s}\theta_{+}^{\prime}(r)\,dr\ge s\theta_{+}^{\prime}(s).
\]
However, it is easy to see that from $\psi=\theta(|\cdot|)$, it follows
that $\langle\eta,\zeta\rangle\le|\zeta|\theta_{+}^{\prime}(|\zeta|)$,
for all $\zeta\in\mathbb{{R}}^{d}$ and $\eta\in\phi(\zeta)$, cf.
\cite[Chapter II, Proposition 8.6]{Show}. Hence the claim follows.
\end{proof}
For the statement of the following lemma and for use further below,
we shall introduce the notion of the so-called \emph{Moreau-Yosida
approximation} $\{\psi^{\lambda}\}_{\lambda>0}$ of $\psi$, that
is, the family of continuous convex functions defined by the following
variational formula
\[
\psi^{\lambda}(\zeta):=\inf_{\eta\in\R^{d}}\left[\psi(\eta)+\frac{1}{2\lambda}|\zeta-\eta|^{2}\right],\quad\zeta\in\R^{d},\;\lambda>0,
\]
see \cite[p. 266]{A} or \cite[p. 97]{Barbu:2012tq} for further details.
\begin{lem}
\label{lem:(N)-lemma} Assume that (N) holds and let $\psi^{\lambda}$,
$\lambda>0$ be the Moreau-Yosida approximation of $\psi$. Then,
there exists a constant $C>0$ (not depending on $\lambda$) such
that
\[
|\psi(\zeta)-\psi^{\lambda}(\zeta)|\le C\lambda(1+\psi(\zeta))\quad\forall\zeta\in\mathbb{R}^{d}.
\]
\end{lem}

\begin{proof}
As above, denote $\phi=\partial\psi$. By the arguments preceding
\cite[Eq. (A.4) in Appendix A]{GessToelle15}, we get in this slightly
more general situation that
\[
|\psi(\zeta)-\psi^{\lambda}(\zeta)|\le\lambda\sup_{\eta\in\phi(\zeta)}|\eta|^{2}\quad\forall\zeta\in\mathbb{R}^{d}.
\]
By \cite[Proof of Proposition 7.1]{GessToelle14} and (N), there exists
a constant $C>0$, such that
\[
|\eta|^{2}\le C(1+\langle\eta,\zeta\rangle)\quad\forall\eta\in\phi(\zeta)\;\forall\zeta\in\mathbb{R}^{d}.
\]
Again, by (N) and by Lemma \ref{lem:Delta2} there exists another
constant $C>0$ such that
\[
\langle\eta,\zeta\rangle\le C\psi(\zeta)\quad\forall\eta\in\phi(\zeta)\;\forall\zeta\in\mathbb{R}^{d}.
\]
Combining these inequalities finishes the proof.
\end{proof}

\subsection{Basic hypotheses on the coefficient fields $a$ and $b$}

We shall formulate the main assumptions on $a$ and $b$.
\begin{assumption}
\label{assu:a-elliptic} Assume that $a\in C^{1}(\T^{d};\R^{d\times d})$
and that there exists a constant $\kappa>0$ such that
\end{assumption}

\[
\tag{E}|a(\xi)\zeta|^{2}\ge\kappa|\zeta|^{2}\text{ for every }\zeta\in\mathbb{R}^{d}\text{ and every }\xi\in\mathbb{T}^{d}.
\]

\begin{assumption}
\label{assu:b-div-free} Assume that $b\in C^{1}(\T^{d};\R^{N\times d})$
such that the rows $b_{i}$, $1\le i\le N$ of $b$ satisfy
\[
\tag{D}\div b_{i}=0\quad\text{on \ensuremath{\T^{d}} for every}\;1\le i\le N.
\]
\end{assumption}

Let $\ell\in\{a,b\}$. Let $L^{\ell}$ denote the Dirichlet operator
associated to the Dirichlet form
\[
\Lcal^{\ell}(u,v):=\int_{\mathbb{T}^{d}}\langle\ell\nabla u,\ell\nabla v\rangle\,d\xi,\quad u,v\in H^{1}(\mathbb{T}^{d}),
\]
where $\ell z$ denotes the application of matrix-multiplication for
$z\in\mathbb{R}^{d}$. For smooth functions $u\in D(L^{\ell})\cap C^{\infty}(\mathbb{T}^{d})$,
by $C^{1}$-regularity of the coefficients, we have that $L^{\ell}u=\operatorname{div}(\ell^{\ast}\ell\nabla u)$,
where $\ell^{\ast}$ denotes the matrix-adjoint of $\ell$. Denote
the associated Dirichlet operators by $L^{a}$, $L^{b}$, respectively,
with Dirichlet forms $\Acal:=\Lcal^{a}$, $\Bcal:=\Lcal^{b}$, respectively.
\begin{rem}
Note that Assumption \ref{assu:a-elliptic} implies that $D(\Lcal^{a})=D(\Acal)=D(\mathcal{E})=S=H^{1}(\T^{d})$
and that $-L^{a}$ is uniformly elliptic with domain equal to $D(-L^{a})=H^{2}(\T^{d})$.
Note also that $b$ and thus $L^{b}$ may be degenerate.
\end{rem}

\begin{lem}
\label{lem:compact_lemma} Suppose that $a\in C^{1}(\T^{d};\R^{d\times d})$.
Then condition (E) is equivalent to $a_{1}(\xi),\ldots,a_{d}(\xi)$
being linearly independent in $\mathbb{R}^{d}$ for each $\xi\in\mathbb{T}^{d}$.
\end{lem}

\begin{proof}
First note that $a^{\ast}a$ is precisely the Gram matrix of the vectors
$a_{1},\ldots,a_{d}$. Consider the statement of $a_{1},\ldots,a_{d}$
being pointwise linearly independent. It is well-known that this is
equivalent to the Gram matrix being positive definite, see e.g. \cite[Chapter 10]{lax2007linear},
which in turn is equivalent to (E) with $\kappa:=\min_{\xi\in\mathbb{T}^{d}}$
$\kappa(\xi)$, $\kappa(\xi)$ being the smallest eigenvalue of $a^{\ast}(\xi)a(\xi)$
(all eigenvalues are strictly positive and real). By continuous dependence
of the eigenvalues of $a^{\ast}a$ on the space variable $\xi\in\mathbb{T}^{d}$,
noting that $\mathbb{T}^{d}$ is compact, we argue by contradiction
to see that $\kappa>0$.
\end{proof}

\subsection{Curvature-dimension condition}

In all of this subsection, assume condition (E) and (N). Set $M:=\T^{d}$
and set $g_{a}:=(a^{\ast}a)^{-1}$. By condition (E), $(M,g_{a})$
is a Riemannian manifold which is \emph{quasi-isometric}\footnote{See \cite[p. 93]{Grigoryan:2009tb} for this notion.}
to $M$, when equipped with the flat metric. We write $M^{a}$ if
we want to emphasize the choice of the metric. We denote the volume
measure on $M^{a}$ by $d\nu:=\sqrt{\det g_{a}}\,d\xi$. Note that
$(M^{a},g_{a})$, equipped with the Lebesgue measure, is thus a \emph{weighted
manifold} $(M^{a},g_{a},d\xi)$ with density $\rho_{a}:=\sqrt{\det(a^{\ast}a)}$
with respect to $\nu$. Note that by \cite[Exercise 3.12]{Grigoryan:2009tb},
$L^{a}=\Delta^{a}$ on the weighted manifold $(M^{a},g_{a},d\xi)$,
where $\Delta^{a}$ denotes the \emph{weighted Laplace-Beltrami operator}
of $M^{a}$.
\begin{defn}
\label{def:BEdefi} Let $\Lambda^{a}:=\{f\in S\;:\:L^{a}f\in S\}$.
We say that $(M^{a},g_{a},d\xi)$ satisfies a \emph{Bakry-\'{E}mery
curvature-dimension condition }$BE(K,\infty)$ if there exists $K\in\R$
with
\[
\tag{BE}L^{a}|a\nabla f|^{2}-2\langle a\nabla f,a\nabla L^{a}f\rangle\ge\frac{K}{2}|a\nabla f|^{2},\quad\forall f\in\Lambda^{a}.
\]
\end{defn}

\begin{prop}
\label{prop:BEcond} Suppose that (E) holds. Suppose that $a\in C^{2}(\T^{d};\R^{d\times d})$
and that for all $1\le i,j\le d$,
\begin{equation}
\sum_{k=1}^{d}\sum_{q=1}^{d}\left[a_{qj}\partial_{k}a_{qi}+a_{qi}\partial_{k}a_{qj}\right]=0\quad\text{on }\T^{d},\label{eq:BEsufficient}
\end{equation}
where $a=(a_{ij})$. Then condition (BE) holds for some $K\le0$.
\end{prop}

\begin{proof}
See Appendix \ref{sec:Proof}.
\end{proof}
Note that condition \eqref{eq:BEsufficient} is symmetric in the sense
that interchanging the indices $i$ and $j$ yields the same condition.
\begin{thm}
\label{thm:ricci-thm} Let $(P_{t}^{a})_{t\ge0}$ be the heat semigroup
associated to $\Acal$. Then the following condition is equivalent
to condition $BE(K,\infty)$.

There exists a constant $K\in\R$, such that
\begin{equation}
|a\nabla P_{t}^{a}f|\le e^{-2Kt}P_{t}^{a}|a\nabla f|\quad\forall t\ge0\;\forall f\in C^{1}(M^{a}).\label{eq:gradient-estimate-A}
\end{equation}
\end{thm}

\begin{proof}
See e.g. \cite{Ambrosio:2015gp,Bakry:2014ir}.
\end{proof}
See \cite{Ambrosio:2015gp,Bakry:2014ir,Wang,Wang:2011fr,Gigli:2013hg,WangYan13,vonRenesse:2005cg}
for the terminology and further results on equivalent \emph{curvature-dimension
conditions} as well as \emph{Ricci curvature bounds} in weighted Riemannian
manifolds.
\begin{lem}
\label{lem:resolvent-contraction-lemma} Let condition (N) and (E)
hold. Let $(P_{t}^{a})_{t\ge0}$ be the heat semigroup associated
to $\Acal$. Assume that $BE(K,\infty)$ holds for some $K\le0$.
Let $J_{\delta}^{0}:=(1-\delta(L^{a}+2K))^{-1}$, $\delta>0$ be the
resolvent associated to $\Acal_{-2K}=\Acal-2K(\cdot,\cdot)_{H}$.
Then
\[
\tilde{\Psi}(J_{\delta}^{0}u)\le\tilde{\Psi}(u)
\]
for any $u\in S$ and any $\delta>0$, where $\tilde{\Psi}$ is as
in \eqref{eq:Psi-tilde}.
\end{lem}

\begin{proof}
Compare with \cite[Proof of Example 7.11]{GessToelle14}. Let $(P_{t}^{0})_{t\ge0}$
be the $C_{0}$-semigroup associated to $J_{\delta}^{0}$, $\delta>0$.
Then by the Trotter product formula \cite[Ch. VIII.8]{ReSi1}, $P_{t}^{0}=e^{2Kt}P_{t}^{a}$,
$t\ge0$. We get for $u\in C^{1}(\T^{d})$ and $\delta>0$ that
\[
\begin{aligned}\left|a\nabla J_{\delta}^{0}u\right|= & \left|a\nabla\int_{0}^{\infty}e^{-t}P_{\delta t}^{0}u\,dt\right|=\left|a\nabla\int_{0}^{\infty}e^{-t+2K\delta t}P_{\delta t}^{a}u\,dt\right|\\
\le & \int_{0}^{\infty}e^{-t+2K\delta t}\left|a\nabla P_{\delta t}^{a}u\right|\,dt\le\int_{0}^{\infty}e^{-t}P_{\delta t}^{a}\left|a\nabla u\right|\,dt=J_{\delta}^{a}|a\nabla u|,
\end{aligned}
\]
where we have used Theorem \ref{thm:ricci-thm}. Since $J_{\delta}^{a}$,
$\delta>0$ is Markovian symmetric on $L^{2}(\T^{d})$, we get by
an application of \cite[Theorem 3]{Mali} and condition (N) that for
$u\in L^{2}(\T^{d})$,
\[
\int_{\T^{d}}\theta(J_{\delta}^{a}u)\,d\xi\le\int_{\T^{d}}\theta(u)\,d\xi.
\]
Altogether, 
\[
\begin{aligned}\tilde{\Psi}(J_{\delta}^{0}u)= & \int_{\T^{d}}\theta(|a\nabla J_{\delta}^{0}u|)\,d\xi\\
\le & \int_{\T^{d}}\theta(J_{\delta}^{a}|a\nabla u|)\,d\xi\le\int_{\T^{d}}\theta(|a\nabla u|)\,d\xi=\tilde{\Psi}(u),
\end{aligned}
\]
density of $C^{1}(\T^{d})\subset S$ and Lebesgue's dominated convergence
theorem completes the proof.
\end{proof}
\begin{rem}
\label{rem:Note-that} Note that
\begin{enumerate}
\item the above statement remains true if $BE(K,\infty)$ holds for $K>0$.
If one investigates the proof carefully, one sees that then it even
holds that
\[
\tilde{\Psi}(J_{\delta}^{a}u)\le\tilde{\Psi}(u)
\]
for any $u\in S$ and any $\delta>0$;
\item the above statement remains true if $\tilde{\Psi}$ is replaced by
$u\mapsto\int_{\mathbb{T}^{d}}\psi_{\lambda}(a\nabla u)\,d\xi$, $u\in S$
as $\psi_{\lambda}$ satisfies condition (N) whenever $\psi$ does
so; for instance, the doubling condition for $\psi_{\lambda}$ holds
with $K$ replaced by $K\vee4$.
\end{enumerate}
\end{rem}

\subsection{Weak defective commutation condition}

In all of this subsection assume conditions (E), (D), and the following.
\begin{assumption}
Assume that $a\in C^{2}(\T^{d};\R^{d\times d})$, $b\in C^{2}(\T^{d};\R^{N\times d})$
and that for all $1\le l,j\le d$, $1\le i\le N$,
\[
\tag{R}\sum_{k=1}^{d}\sum_{p=1}^{d}\left[b_{ik}(a_{pl}\partial_{k}a_{pj}+a_{pj}\partial_{k}a_{pl})-a_{pk}(a_{pj}\partial_{k}b_{il}+a_{pl}\partial_{k}b_{ij})\right]=0\quad\text{on }\T^{d},
\]
where $a=(a_{lj})$ and $b=(b_{ij})$.
\end{assumption}

Note that condition (R) is symmetric in the sense that interchanging
the indices $l$ and $j$ yields the same condition.
\begin{prop}
\label{prop:WDC} Conditions (E), (R), and (D) imply the following
\emph{weak defective commutation condition}.

There exists a constant $c\in\mathbb{R}$ such that for every $\beta>0$,
we have that
\begin{equation}
\beta\int_{\mathbb{T}^{d}}\left\langle \beta G_{\beta}^{a}b\nabla f-\beta b\nabla G_{\beta}^{a}f,b\nabla f\right\rangle \,d\xi\ge c\Acal(f,f),\quad\forall f\in S.\label{eq:WDC}
\end{equation}
\end{prop}

As seen later in the proof of Theorem \ref{thm:main-bound} below
(where inequality \eqref{eq:WDC} is applied), we could allow for
replacing the term $c\Acal(f,f)$ on the l.h.s. of \eqref{eq:WDC}
by $c\Acal_{1}(f,f)$.

In fact, by examining the proof below, we get that $c\le0$.
\begin{proof}[Proof of Proposition \ref{prop:WDC}.]

Let $f\in C^{3}(\T^{d})$. By assumption, $a\in C^{2}(\T^{d};\R^{d\times d})$
and $b\in C^{2}(\T^{d};\R^{N\times d})$. Set $\A:=a^{\ast}a$. Utilizing
the Einstein summation convention, we get that for every $1\le i\le N$,
\begin{equation}
\begin{aligned} & (b\nabla(\div(\A\nabla f))_{i}\\
= & b_{ij}\partial_{j}(\partial_{k}(\A_{kl}\partial_{l}f))\\
= & b_{ij}\left[\partial_{j}\partial_{k}\A_{kl}\partial_{l}f+\partial_{k}\A_{kl}\partial_{j}\partial_{l}f+\partial_{j}\A_{kl}\partial_{k}\partial_{l}f+\A_{kl}\partial_{j}\partial_{k}\partial_{l}f\right]\\
= & b_{ij}\partial_{j}\partial_{k}\A_{kl}\partial_{l}f+\partial_{k}\A_{kl}\partial_{l}(b_{ij}\partial_{j}f)-\partial_{k}\A_{kl}\partial_{l}b_{ij}\partial_{j}f\\
 & +\A_{kl}\partial_{l}\partial_{k}(b_{ij}\partial_{j}f)-\A_{kl}\partial_{l}\partial_{k}b_{ij}\partial_{j}f\\
 & -\A_{kl}\partial_{k}b_{ij}\partial_{l}\partial_{j}f-\A_{kl}\partial_{l}b_{ij}\partial_{k}\partial_{j}f+b_{ij}\partial_{j}\A_{kl}\partial_{k}\partial_{l}f\\
= & \div(\A\nabla(b\nabla f)_{i})+\left[b_{il}\partial_{l}\partial_{k}\A_{kj}-\partial_{k}\A_{kl}\partial_{l}b_{ij}-\A_{kl}\partial_{l}\partial_{k}b_{ij}\right]\partial_{j}f\\
 & +\left[b_{ik}\partial_{k}(a_{pj}a_{pl})-2a_{pk}a_{pj}\partial_{k}b_{il}\right]\partial_{l}\partial_{j}f\\
= & \div(\A\nabla(b\nabla f)_{i})+(R_{0}\nabla f)_{i}
\end{aligned}
\label{eq:LaLaR}
\end{equation}
by relabeling the indices, symmetry of the Hessian and condition (R),
where
\[
(R_{0}z)_{i}:=[b_{il}\partial_{l}\partial_{k}\A_{kj}-\partial_{k}\A_{kl}\partial_{l}b_{ij}-\A_{kl}\partial_{l}\partial_{k}b_{ij}]z_{j},\quad z\in\R^{d},\;1\le i\le N.
\]
Clearly, $f\mapsto R_{0}\nabla f$ can uniquely be extended to a linear
operator $R:H^{1}(\T^{d})\to L^{2}(\T^{d};\R^{N})$ such that there
exists $C=C(a,\nabla a,D^{2}a,b,\nabla b,D^{2}b)>0$ with 
\begin{equation}
\|Ru\|_{L^{2}(\T^{d};\R^{N})}^{2}\le C\Acal(u,u)\quad\forall u\in S.\label{eq:R-bounded-op}
\end{equation}
Note that $R$ is bounded as an operator from $S$ to $L^{2}(\T^{d};\R^{N})$.
By \cite[Theorem 2.1]{Shigekawa:2006dp}, then \eqref{eq:LaLaR} implies
that for every $\beta>0$,
\begin{equation}
b\nabla G_{\beta}^{a}u=G_{\beta}^{a}b\nabla u+G_{\beta}^{a}RG_{\beta}^{a}u\quad\forall u\in S.\label{eq:defective}
\end{equation}
See also \cite[Theorem 3.1, Proposition 3.2]{Shigekawa:2000io}. Let
$u\in S$ and $\beta>0$. Clearly, \eqref{eq:defective} implies that
\begin{equation}
\langle b\nabla G_{\beta}^{a}u,b\nabla u\rangle=\langle G_{\beta}^{a}b\nabla u,b\nabla u\rangle+\langle G_{\beta}^{a}RG_{\beta}^{a}u,b\nabla u\rangle.\label{eq:indeed_cond}
\end{equation}
Integrating over $\T^{d}$ and multiplying with $\beta^{2}$ yields
\[
\begin{aligned} & \beta\int_{\mathbb{T}^{d}}\left\langle \beta G_{\beta}^{a}b\nabla u-\beta b\nabla G_{\beta}^{a}u,b\nabla u\right\rangle \,d\xi\\
= & -\int_{\T^{d}}\langle R\beta G_{\beta}^{a}u,\beta G_{\beta}^{a}b\nabla u\rangle\,d\xi\\
\ge & -\sqrt{C}\Acal(\beta G_{\beta}^{a}u,\beta G_{\beta}^{a}u)^{1/2}\|\beta G_{\beta}^{a}b\nabla u\|_{L^{2}(\T^{d};\R^{d})}\\
\ge & -\sqrt{C}\Acal(u,u)^{1/2}\|b\nabla u\|_{L^{2}(\T^{d};\R^{d})}\\
= & -\sqrt{C}\Acal(u,u)^{1/2}\Bcal(u,u)^{1/2}\\
\ge & -\sqrt{C}\|b\|_{\infty}\kappa^{-1/2}\Acal(u,u)
\end{aligned}
\]
which yields \eqref{eq:WDC}.
\end{proof}
The (weak) defective commutation property is a variant of the so-called
defective intertwining property, see \cite{Shigekawa:2006dp}.

Now, our conditions imply the following technical result which is
needed later.
\begin{lem}
\label{lem:comm-implies-M} Suppose that conditions (E), (R), and
(D) hold. Then the resolvent $J_{\delta}^{a}$, $\delta>0$ leaves
the domain $D(L^{b})$ of $L^{b}$ invariant and we have that $L^{b}J_{\delta}^{a}z\rightharpoonup L^{b}z$
weakly in $H$ as $\delta\to0$ for any $z\in D(L^{b})$.
\end{lem}

\begin{proof}
For $u\in D(L^{b})$, $v\in S$, $\delta>0$, consider
\[
\begin{aligned} & \left|(L^{b}J_{\delta}^{a}u-L^{b}u,v)_{H}\right|\\
\le & \left|\int_{\T^{d}}\langle b\nabla(u-J_{\delta}^{a}u),b\nabla v\rangle\,d\xi\right|\\
\le & \left|\int_{\T^{d}}\langle b\nabla u-J_{\delta}^{a}b\nabla u,b\nabla v\rangle\,d\xi\right|+\delta\left|\int_{\T^{d}}\langle J_{\delta}^{a}RJ_{\delta}^{a}u,b\nabla v\rangle\,d\xi\right|\\
\le & \left|\int_{\T^{d}}\langle b\nabla u-J_{\delta}^{a}b\nabla u,b\nabla v\rangle\,d\xi\right|+\delta\left|\int_{\T^{d}}\langle RJ_{\delta}^{a}u,b\nabla J_{\delta}^{a}v-\delta J_{\delta}^{a}RJ_{\delta}^{a}v\rangle\,d\xi\right|,
\end{aligned}
\]
where we have used \eqref{eq:defective} from the proof of Proposition
\ref{prop:WDC} twice. $R$ is defined as in \eqref{eq:R-bounded-op}.
The first term converges to zero as $\delta\to0$. The second term
is bounded for $\delta\in(0,1]$ as follows
\[
\begin{aligned} & \delta C(R)\Acal(J_{\delta}^{a}u,J_{\delta}^{a}u)^{1/2}[\kappa^{-1/2}\|b\|_{\infty}+\delta]\Acal(J_{\delta}^{a}v,J_{\delta}^{a}v)^{1/2}\\
\le & C(R,a,b)\|L^{b}u\|_{H}\Acal(\delta J_{\delta}^{a}v,\delta J_{\delta}^{a}v)^{1/2}\\
\le & \sqrt{\delta}C(R,a,b)\|L^{b}u\|_{H}\|J_{\delta}^{a}v-v\|_{H}^{1/2}\|v\|_{H}^{1/2},
\end{aligned}
\]
and hence converges to zero as $\delta\to0$. Now, as we have the
bound
\[
\|L^{b}J_{\delta}^{a}u\|_{H}\le C(R,a,b)\|L^{b}u\|_{H}(1+\|u\|_{H}),
\]
by an $\eps/2$-argument, we get the convergence to zero as $\delta\to0$
for every $v\in H$. As a consequence, $L^{b}J_{\delta}^{a}u\rightharpoonup L^{b}u$
weakly in $H$ for $u\in D(L^{b})$.
\end{proof}
\begin{rem}
If one lets $\beta\to\infty$ in \eqref{eq:WDC} for $f\in C^{3}(\mathbb{T}^{d})$,
one obtains
\begin{equation}
\int_{\mathbb{T}^{d}}\left[\langle L^{a}b\nabla f,b\nabla f\rangle-\langle b\nabla f,b\nabla L^{a}f\rangle\right]\,d\xi\ge c\Acal(f,f),\quad\forall f\in C^{3}(\mathbb{T}^{d}),\label{eq:integratedBE}
\end{equation}
which resembles an integrated version of the \emph{Bakry-\'{E}mery
curvature-dimension condition}, see (BE) above, see \cite{Bakry:1997tf,Ledoux:2000vx,Gigli:2013hg}
for similar kinds of conditions. Assume that (D), (E), and (R) hold
and that $\Acal$ satisfies the following Poincar\'{e} inequality,
i.e, there exists $C>0$ such that 
\begin{equation}
\|f\|_{L^{2}(\T^{d})}^{2}\le C\Acal(f,f)+\left(\int_{\T^{d}}f\,d\xi\right)^{2}\quad\forall f\in S,\label{eq:poincare}
\end{equation}
and $C:=\lambda_{1}^{-1}$ is optimal, where $\lambda_{1}>0$ is the
smallest non-zero eigenvalue of $-L^{a}$, cf. \cite[Example 1.1.2]{Wang}.
Observe that then there exists a non-constant function $f\in S$ such
that $L^{a}f=-\lambda_{1}f$. For such an eigenfunction $f$, we get
by the Poincar\'{e} inequality \eqref{eq:poincare} that
\begin{equation}
\begin{aligned}\int_{\mathbb{T}^{d}}\left[\langle L^{a}b\nabla f,b\nabla f\rangle-\langle b\nabla f,b\nabla L^{a}f\rangle\right]\,d\xi= & -\sum_{i=1}^{N}\Acal\left(\langle b_{i},\nabla f\rangle,\langle b_{i},\nabla f\rangle\right)+\lambda_{1}\int_{\mathbb{T}^{d}}|b\nabla f|^{2}\,d\xi\\
\le & \lambda_{1}\sum_{i=1}^{N}\left(\int_{\mathbb{T}^{d}}\langle b_{i},\nabla f\rangle\,d\xi\right)^{2}=0,
\end{aligned}
\label{eq:example_for_b}
\end{equation}
where, in the last step, we have used integration by parts and (E).
Hence in the situation of a Poincar\'{e} inequality for $\Acal$ and
for general $b$ satisfying (D), the estimate \eqref{eq:example_for_b},
combined with \eqref{eq:integratedBE}, shows that $c\le0$ is necessary
for \eqref{eq:WDC} to hold.

Now, we shall prove an approximative commutator bound which is needed
later.
\end{rem}

\begin{lem}
\label{lem:commutator-estimate} Assume that (E), (R), and (D) hold.
Let $u\in H$. Then there exists $c\in\R$ such that for every $\delta>0$
and every $\alpha\ge0$, $\beta>0$,
\[
\sum_{i=1}^{N}\Acal_{\alpha}^{(\beta)}(\langle b_{i},\nabla J_{\delta}^{a}u\rangle,\langle b_{i},\nabla J_{\delta}^{a}u\rangle)+\Acal_{\alpha}^{(\beta)}(J_{\delta}^{a}L^{b}J_{\delta}^{a}u,u)\le-c\Acal(J_{\delta}^{a}u,J_{\delta}^{a}u).
\]
\end{lem}

\begin{proof}
Let $u\in H$, let $\delta,\beta>0$, $\alpha\ge0$. Set $y_{\delta}:=J_{\delta}^{a}u$.
Then by Proposition \ref{prop:WDC} there exists $c\in\R$, such that
\[
\begin{aligned} & \sum_{i=1}^{N}\Acal_{\alpha}^{(\beta)}\left(\langle b_{i},\nabla J_{\delta}^{a}u\rangle,\langle b_{i},\nabla J_{\delta}^{a}u\rangle\right)+\Acal_{\alpha}^{(\beta)}\left(J_{\delta}^{a}L^{b}J_{\delta}^{a}u,u\right)\\
= & \beta\left[\Bcal(\beta G_{\beta}^{a}y_{\delta},y_{\delta})-\sum_{i=1}^{N}\left(\beta G_{\beta}^{a}(\langle b_{i},\nabla y_{\delta}\rangle),\langle b_{i},\nabla y_{\delta}\rangle\right)_{H}\right]\\
= & \beta\int_{\mathbb{T}^{d}}\left\langle \beta b\nabla G_{\beta}^{a}y_{\delta}-\beta G_{\beta}^{a}b\nabla y_{\delta},b\nabla y_{\delta}\right\rangle \,d\xi\\
\le & -c\Acal(y_{\delta},y_{\delta}).
\end{aligned}
\]
\end{proof}
\begin{cor}
\label{cor:forA3} Assume that (E), (R), and (D) hold. Let $u\in S$.
Then the previous lemma implies
\[
\Acal_{\alpha}^{(\beta)}(u,J_{\delta}^{a}L^{b}J_{\delta}^{a}u)\le(-c\vee0)\Acal(u,u)\quad\forall\delta,\beta>0,
\]
where $c\in\R$ does neither depend on $\delta$, nor on $\beta$.
\end{cor}

\subsection{\label{subsec:Examples} Examples}

The simplest example is of course that $d=N$ and $a\not=0$ as well
as $b$ both equal the $d\times d$-identity matrix times some real
constants. We also note that for $d=1$, condition (D) implies that
$b$ equals a real constant, so that combined with condition (R) it
follows that $a$ has to equal a real constant which has to be non-zero
by condition (E).
\begin{example}
\label{exa:comm-killing} Suppose that $a\equiv(\delta_{ij})$ (i.e.,
the Kronecker delta / identity matrix), and let $b\in C^{2}(\T^{d};\R^{N\times d})$
such that

\begin{equation}
\partial_{j}b_{il}+\partial_{l}b_{ij}=0\quad\forall1\le l,j\le d\;\forall1\le i\le N.\label{eq:killing}
\end{equation}
It is easily seen that (D) is satisfied. However, after contracting
$a$ twice in the formula for (R), we see that (R) is indeed satisfied.
Equality \eqref{eq:killing} is related to the notion of so-called
\emph{Killing vector fields}, see \cite{Ciotir:2016fe} for examples
in a related context.
\end{example}

\begin{example}
Suppose that $d=N=2$ and let $\xi=(t,s)\in\mathbb{T}^{2}$.
\[
a(\xi)=\begin{pmatrix}a_{1}(t,s)\\
a_{2}(t,s)
\end{pmatrix}:=\begin{pmatrix}h_{1}(t-s) & 0\\
0 & h_{2}(t-s)
\end{pmatrix},\quad b:\equiv\begin{pmatrix}1 & 1\\
1 & 1
\end{pmatrix},
\]
with $h_{1},h_{2}\in C^{2}(\T)$, $h_{1}>0$, $h_{2}>0$. Then (E),
(D) and (R) are satisfied. We shall omit the computations.
\end{example}

\section{\label{sec:Stochastic-VI} Stochastic variational inequalities (SVI)}

As above, $H=L^{2}(\mathbb{T}^{d})$, $S=H^{1}(\mathbb{T}^{d})$.
Let $U:=\mathbb{R}^{N}$. Denote by $L_{2}(U,H)$ the space of linear
Hilbert-Schmidt operators from $U$ to $H$. Let $B:S\to L_{2}(U,H)$
denote the linear operator 
\[
B(u)\zeta:=\sum_{i=1}^{N}\langle b_{i},\nabla u\rangle\zeta^{i},\quad u\in S,\;\zeta\in U.
\]

\begin{rem}
\label{rem:hilbert-schmidt}By Lemma \ref{lem:comm-implies-M}, $B$
is bounded from $S$ to $L_{2}(U,H)$ and its norm is given by
\[
\|B(u)\|_{L_{2}(U,H)}^{2}=\Bcal(u,u),\quad u\in S.
\]
\end{rem}

\begin{defn}
\label{def:svi-new} Let $x\in L^{2}(\Omega,\Fcal_{0},\P;H)$, $T>0$.
A progressively measurable\footnote{That is, for every $t\in[0,T]$ the map $X:[0,t]\times\Omega\to H$
is $\mathcal{B}([0,t])\otimes\mathcal{F}_{t}$-measurable.} map $X\in L^{2}([0,T]\times\Omega;H)$ is said to be an \emph{SVI-solution}
to \eqref{eq:main-spde-1} if there exists a constant $C>0$ such
that
\begin{enumerate}
\item \emph{(Regularity)}
\begin{equation}
\operatorname{ess\;sup}\displaylimits_{t\in[0,T]}\mathbb{E}\|X_{t}\|_{H}^{2}+2\mathbb{E}\int_{0}^{T}\Psi(X_{s})\,ds\le\E\|x\|_{H}^{2}.\label{eq:SVI_regularity-Stratonovich}
\end{equation}
\item \emph{(Variational inequality)} For every choice of \emph{admissible
test-elements} \emph{$(Z_{0},Z,G,P)$,} that is, by definition, $Z_{0}\in L^{2}(\Omega,\Fcal_{0},\P;S)$,
$Z\in L^{2}([0,T]\times\Omega;D(L^{b})\cap S)$, $G\in L^{2}([0,T]\times\Omega;H)$,
$P\in L(H)$ such that $G$ is progressively measurable, such that
$P(D(L^{b}))\subset D(L^{b})$ and such that the following equation
holds in It\^{o}-sense
\[
Z_{t}=Z_{0}+\int_{0}^{t}G_{s}\,ds+\frac{1}{2}\int_{0}^{t}P^{\ast}L^{b}PZ_{s}\,ds+\int_{0}^{t}BPZ_{s}\,dW_{s}\quad\forall t\in[0,T],
\]
we have that
\begin{equation}
\begin{aligned} & \E\|X_{t}-Z_{t}\|_{H}^{2}+2\E\int_{0}^{t}\Psi(X_{s})\,ds\\
\le & \E\|x-Z_{0}\|_{H}^{2}+2\E\int_{0}^{t}\Psi(Z_{s})\,ds\\
 & -2\E\int_{0}^{t}(G_{s},X_{s}-Z_{s})_{H}\,ds\\
 & -\E\int_{0}^{t}(L^{b}PZ_{s},PX_{s}-X_{s})_{H}\,ds-\E\int_{0}^{t}(X_{s},L^{b}(Z_{s}-PZ_{s}))_{H}\,ds
\end{aligned}
\label{eq:SVI-with-Dirichlet-operator}
\end{equation}
for almost all $t\in[0,T]$.
\end{enumerate}
If, additionally, it holds that $X\in C([0,T];L^{2}(\Omega;H))$,
we say that $X$ is a \emph{(time-)continuous SVI-solution to \eqref{eq:main-spde-1}.}
\end{defn}

\begin{defn}
We say that an $\{\Fcal_{t}\}$-adapted process $X\in L^{2}(\Omega,C([0,T];H))\cap L^{2}([0,T]\times\Omega;D(L^{b})\cap S)$
is an \emph{analytically strong solution }to \eqref{eq:main-spde-1}
with initial datum $x\in L^{2}(\Omega;H)$, if there exists $\eta\in L^{2}([0,T]\times\Omega;H)$
such that $\eta$ is progressively measurable and such that $\eta\in\partial\Psi(X)$
a.e. and we have that
\begin{equation}
X_{t}=x-\int_{0}^{t}\eta_{s}\,ds+\frac{1}{2}\int_{0}^{t}L^{b}X_{s}\,ds+\int_{0}^{t}\langle b\nabla X_{s},dW_{s}\rangle\quad\P\text{-a.s.}\label{eq:strong-sln}
\end{equation}
for all $t\ge0$.
\end{defn}

\begin{lem}
\label{lem:stronglemma} If $X$ is an analytically strong solution
to \eqref{eq:main-spde-1}, then $X$ is a time-continuous SVI-solution
to \eqref{eq:main-spde-1}.
\end{lem}

\begin{proof}
Compare with \cite[Proof of Remark 2.3]{GessToelle15}. Let $X$ be
a strong solution to \eqref{eq:formal-SPDE}. Then by It\^{o}'s formula
and a standard localization argument, compare with Remark \ref{rem:hilbert-schmidt}:
\[
\E\|X_{t}\|_{H}^{2}=\E\|x\|_{H}^{2}-2\E\int_{0}^{t}(\eta_{s},X_{s})_{H}\,ds+\E\int_{0}^{t}(L^{b}X_{s},X_{s})_{H}\,ds+\E\int_{0}^{t}\Bcal(X_{s},X_{s})\,ds.
\]
By the definition of the subdifferential $\partial\Psi$, we get that
\[
(-\eta,X)_{H}=(\eta,0-X)\le\Psi(0)-\Psi(X)=-\Psi(X)\quad dt\otimes\P\text{-a.e.}
\]
Hence \eqref{eq:SVI_regularity-Stratonovich} is satisfied.

Furthermore, let $(Z_{0},Z,G,P)$ be a fixed choice of admissible
test-elements. We have that 
\[
Z_{t}=Z_{0}+\int_{0}^{t}G_{s}\,ds+\frac{1}{2}\int_{0}^{t}P^{\ast}L^{b}PZ_{s}\,ds+\int_{0}^{t}B(PZ_{s})\,dW_{s}\quad\forall t\in[0,T].
\]
Consider
\[
d(X-Z)=(-\eta-G)\,dt+\frac{1}{2}(L^{b}X-P^{\ast}L^{b}PZ)\,dt+B(X-PZ)\,dW
\]
and we get by It\^{o}'s formula that for $t\in[0,T]$,
\[
\begin{aligned}\|X_{t}-Z_{t}\|_{H}^{2}= & \|x-Z_{0}\|_{H}^{2}+2\int_{0}^{t}(-\eta_{s}-G_{s},X_{s}-Z_{s})_{H}\,ds\\
 & +\int_{0}^{t}(L^{b}X_{s}-P^{\ast}L^{b}PZ_{s},X_{s}-Z_{s})_{H}\,ds\\
 & +2\int_{0}^{t}(X_{s}-Z_{s},B(X_{s}-PZ_{s})\,dW_{s})_{H}\\
 & +\int_{0}^{t}\|B(X_{s}-PZ_{s})\|_{L_{2}(U,H)}^{2}.
\end{aligned}
\]
Taking expectations yields for $t\in[0,T]$,
\[
\begin{aligned}\E\|X_{t}-Z_{t}\|_{H}^{2}= & \E\|x-Z_{0}\|_{H}^{2}+2\E\int_{0}^{t}(-\eta_{s}-G_{s},X_{s}-Z_{s})_{H}\,ds\\
 & +\E\int_{0}^{t}(L^{b}X_{s}-P^{\ast}L^{b}PZ_{s},X_{s}-Z_{s})_{H}\,ds\\
 & +\E\int_{0}^{t}\Bcal(X_{s}-PZ_{s},X_{s}-PZ_{s})\,ds\\
= & \E\|x-Z_{0}\|_{H}^{2}+2\E\int_{0}^{t}(-\eta_{s}-G_{s},X_{s}-Z_{s})_{H}\,ds\\
 & +\E\int_{0}^{t}\Bcal(X_{s},Z_{s}-PZ_{s})\,ds+\E\int_{0}^{t}\Bcal(PZ_{s},PX_{s}-X_{s})\,ds.
\end{aligned}
\]
The proof is completed by the application of symmetry of $L^{b}$
in $H$ and of the subdifferential property for $\eta\in\partial\Psi(X)$
\[
(-\eta,X-Z)\le\Psi(Z)-\Psi(X)\quad dt\otimes\P\text{-a.e.}
\]
\end{proof}
\begin{rem}
In \cite{Ciotir:2016fe}, a weaker notion for a solution for an analogue
of equation \eqref{eq:formal-SPDE} on a bounded, convex domain with
Neumann boundary conditions was proposed. Translated to our situation,
in \eqref{eq:SVI-with-Dirichlet-operator}, one would test the SVI
with elements $(Z_{0},Z,G)$ and $P\equiv1$ (the identity) --- making
the notion a weaker one compared to the one of Definition \ref{def:svi-new}.
However, in \cite[Definition 3.1]{Ciotir:2016fe}, less regularity
for $Z$ was demanded. Let us point out again that we are able to
consider more general gradient noise $B$ here.
\end{rem}

\section{\label{sec:The-main-result} The main result and its proof}
\begin{thm}
\label{thm:mainmainthm} Suppose that condition (N) holds for $\psi$
and that conditions (E) and (D) hold. Suppose that condition (BE)
holds for $a$. Suppose that condition (R) holds.

Then, for every initial datum $x\in L^{2}(\Omega,\Fcal_{0},\P;H)$
and every $T>0$, there exists a unique adapted time-continuous SVI-solution
$X\in C([0,T];L^{2}(\Omega;H))$ to the SPDE \eqref{eq:main-spde-1}
such that
\begin{equation}
\operatorname{ess\;sup}\displaylimits_{t\in[0,T]}\E\|X_{t}-Y_{t}\|_{H}^{2}\le\E\|x-y\|_{H}^{2}\label{eq:uniqueness-bound}
\end{equation}
where $Y\in C([0,T];L^{2}(\Omega;H))$ is the unique SVI-solution
to another initial datum $y\in L^{2}(\Omega,\Fcal_{0},\P;H)$.
\end{thm}

The remaining part of this section is devoted to the proof of Theorem
\ref{thm:mainmainthm}.

\subsection{The approximating equation}

Consider the \emph{Gelfand triple} of dense and compact embeddings
\[
S\hookrightarrow H\hookrightarrow S^{\ast}.
\]
Suppose that conditions (E), (D) and condition $BE(K,\infty)$ hold
for $a$ with $K\in\R$. There are two cases. If $K\le0$, we shall
equip $S$ with the equivalent norm $(\Acal_{1-2K})^{1/2}=(\Acal+(1-2K)\|\cdot\|_{H}^{2})^{1/2}$.
If $K>0$, it is sufficient to consider $\Acal_{1}^{1/2}$. Note that
these norms are equivalent to the standard norm of $H^{1}(\T^{d})$
by \cite[Chapter I, Exercise 2.1]{MR} and assumption (E). From now
on, we shall concentrate on the first case $K\le0$, as it is slightly
more involved. We remark that in the latter case $K>0$, we can recover
the a priori bounds proved below from the first case by a standard
perturbation argument.

Let us denote by $\{W_{t}\}_{t\ge0}$ a cylindrical Wiener process
in $U:=\R^{d}$ for the stochastic basis $(\Omega,\mathcal{F},\{\mathcal{F}_{t}\}_{t\ge0},\mathbb{P})$.
Consider the It\^{o}-equation in the above Gelfand triple
\[
dX_{t}+A^{\lambda,\delta,\eps}(X_{t})\,dt=B^{\delta}(X_{t})\,dW_{t},\quad X_{0}=x_{n},
\]
where $x_{n}\in L^{2}(\Omega,\Fcal_{0},\P;S)$, and where, from now
on, whenever it seems convenient, we suppress the indices in the notation
as in $X=X^{n,\lambda,\delta,\varepsilon}$. Here, we define for $\lambda,\delta,\varepsilon>0$:
\begin{equation}
A^{\lambda,\delta,\varepsilon}(x):=-\operatorname{div}(a^{\ast}\phi^{\lambda}(a\nabla x))-\varepsilon L^{a}x-\frac{1}{2}J_{\delta}^{a}L^{b}J_{\delta}^{a}x,\label{eq:Def-A}
\end{equation}
writing $\phi^{\lambda}=\partial\psi^{\lambda}$, $\lambda>0$, meaning
that $\phi^{\lambda}$ is the Yosida-approximation of $\phi$, whereas
$\psi^{\lambda}$ denotes the Moreau-Yosida approximation of $\psi$,
see \cite[p. 266]{A}. Furthermore, for $\zeta=(\zeta^{1},\ldots,\zeta^{d})\in U$,
$x\in S$, $t\in[0,T]$, 
\begin{equation}
B^{\delta}(x)\zeta:=\sum_{i=1}^{d}\langle b_{i},\nabla J_{\delta}^{a}x\rangle\zeta^{i}.\label{eq:Def-B}
\end{equation}

\begin{prop}
\label{prop:conds} Suppose that condition (N) holds for $\psi$ and
that conditions (E) and (D) hold. Suppose that condition (BE) holds
for $a$ with $K\le0$. Suppose that condition (R) holds. Then $A^{\lambda,\delta,\eps}$
and $B^{\delta}$ as defined in \eqref{eq:Def-A}, and \eqref{eq:Def-B}
respectively, satisfy conditions (A1)--(A3) in Appendix \ref{sec:app-ex},
and (B1)--(B2) in Appendix \ref{sec:app-ex} respectively, for $U:=\mathbb{R}^{d}$
and all $\lambda,\delta,\varepsilon>0$.
\end{prop}

\begin{proof}
Fix $\lambda,\delta,\varepsilon>0$.
\begin{enumerate}
\item[(A1):]  The map 
\[
x\mapsto-\operatorname{div}(a^{\ast}\phi^{\lambda}(a\nabla x))-\eps L^{a}x
\]
 is the maximal monotone subdifferential of the l.s.c. map
\[
u\mapsto\int_{\mathbb{T}^{d}}\psi^{\lambda}(a\nabla u)\,d\xi+\frac{\eps}{2}\int_{\T^{d}}|a\nabla u|^{2}\,d\xi.
\]
In this regard, see also \cite[Proposition II.7.8]{Show} for the
chain rule of a convex functional and bounded linear operators. The
map 
\[
x\mapsto-\frac{1}{2}J_{\delta}^{a}L^{b}J_{\delta}^{a}x
\]
is linear and positive definite due to condition (E) and (D). As this
map is also bounded linear, the conditions of Browder's theorem on
the maximality of the sum of monotone operators are satisfied, see
\cite[Theorem, pp. 75--76]{Rockafellar:1970ew}, and hence (A1) holds.
\item[(A2):]  As $A^{\lambda,\delta,\eps}$ is a subpotential operator of a convex
energy which is bounded by $C(1+\|\cdot\|_{S}^{2})$ for some positive
constant $C>0$, condition (A2) easily follows with similar arguments
as in \cite[Proposition 7.1]{GessToelle14}.
\item[(A3):]  The claim for the nonlinear part follows from Lemma \ref{lem:resolvent-contraction-lemma}
and by similar arguments as in \cite[Proposition 7.1]{GessToelle14}.
Compare also to \cite{Ciotir:2016fe,VBMR} for the Dirichlet and Neumann
boundary cases respectively. The detailed argument can also be found
in \eqref{eq:wangtrick} in the proof of Theorem \ref{thm:main-bound}
below. Regarding the second part of the linear part, by Corollary
\ref{cor:forA3}, for every $\mu>0$, $x\in S$,
\[
\begin{aligned} & -(J_{\delta}^{a}L^{b}J_{\delta}^{a}x,\tilde{L}^{(\mu)}x)_{H}=\Acal_{1-2K}^{(1/\mu)}(x,J_{\delta}^{a}L^{b}J_{\delta}^{a}x)\\
\le & (-c\vee0)\Acal(x,x)\le(-c\vee0)\Acal_{1-2K}(x,x).
\end{aligned}
\]
 Note that due to the renorming of $S$, $\tilde{L}^{(\mu)}:=(L^{a}+2K-1)^{(\mu)}$,
plays the role of $L^{(\mu)}$ here.
\item[(B1):]  By a straightforward calculation, we get for $x\in S$,
\[
\|B^{\delta}(x)\|_{L_{2}(U,H)}^{2}=\Bcal(J_{\delta}^{a}x,J_{\delta}^{a}x)\le\|b\|_{\infty}^{2}\kappa^{-1}\Acal_{1-2K}(x,x).
\]
\item[(B2):]  By linearity, for $x,y\in S$,
\[
\begin{aligned} & \|B^{\delta}(x)-B^{\delta}(y)\|_{L_{2}(U,H)}^{2}=\|B^{\delta}(x-y)\|_{L_{2}(U,H)}^{2}\\
= & \Bcal(J_{\delta}^{a}(x-y),J_{\delta}^{a}(x-y))\le\|b\|_{\infty}^{2}\kappa^{-1}\Acal^{(1/\delta)}(x-y,x-y)\\
\le & \left(\frac{2}{\delta}+1\right)\|b\|_{\infty}^{2}\kappa^{-1}\|x-y\|_{H}^{2},
\end{aligned}
\]
compare also with \cite[Chapter I, Lemma 2.11]{MR}.
\end{enumerate}
\end{proof}
As a consequence, (A1)--(A3), (B1)--(B2) guarantee the existence
and uniqueness of limit solutions to \eqref{eq:Ito-multivalued},
see Appendix \ref{sec:app-ex} for the definition and the precise
result.

\subsection{A priori estimates}

As before, consider the single-valued It\^{o}-SPDE
\begin{equation}
dX_{t}+A^{\lambda,\delta,\varepsilon}(X_{t})\,dt=B^{\delta}(X_{t})\,dW_{t},\quad X_{0}=x_{n},\label{eq:approx-eq-B1B2}
\end{equation}
where $x_{n}\in L^{2}(\Omega,\Fcal_{0},\P;S)$, and where, as above,
we suppress the indices in the notation as in $X=X^{n,\lambda,\delta,\varepsilon}$.
The existence and uniqueness of limit solutions to \eqref{eq:approx-eq-B1B2}
follows from \cite[Theorem 4.6]{GessToelle14}, see Appendix \ref{sec:app-ex};.

For fixed $\delta>0$ and for every $m\in\mathbb{N}$, let $y\mapsto B^{\delta.m}(y)$
be progressively measurable maps on $S$ such that
\begin{enumerate}
\item[(i)]  each $B^{\delta,m}$ satisfies (B1)--(B2) with constants $C_{1}$
and $C_{2}$ not depending on $m$,
\item[(ii)]  each $B^{\delta,m}$ satisfies (B3) (with constants $C_{3}=C_{3}(m)$
typically depending on $m$),
\item[(iii)]  $\|B^{\delta,m}(y)-B^{\delta}(y)\|_{L_{2}(U,H)}\to0$ for every
$y\in S$ as $m\to\infty$.
\end{enumerate}
The existence of such a sequence of maps can be proved e.g. by introducing
an approximation step that employs standard mollifiers. Consider the
sequence of It\^{o}-processes
\begin{equation}
X_{t}^{m}=x_{n}-\int_{0}^{t}A^{\lambda,\delta,\varepsilon}(X_{s}^{m})\,ds+\int_{0}^{t}B^{\delta,m}(X_{s}^{m})\,dW_{s},\label{eq:approx-eq-B1B2B3}
\end{equation}
where $x_{n}\in L^{2}(\Omega,\Fcal_{0},\P;S)$, $m\in\mathbb{N}$.
The existence of such processes is guaranteed by \cite[Theorem 4.4]{GessToelle14},
see Appendix \ref{sec:app-ex}. The following proposition is a modification
of \cite[Proposition 14]{Toelle2018proceedings}.
\begin{prop}
\label{prop:H_Bound} Suppose that condition (N) holds and that condition
(E), (D) hold and that condition (BE) holds for $a$ with $K\le0$.
Let $\lambda,\delta,\varepsilon>0$. Let $x_{n}\in L^{2}(\Omega,\Fcal_{0},\P;S)$,
and let $X=X^{n,\lambda,\delta,\varepsilon}$ be a limit solution
to \eqref{eq:approx-eq-B1B2}. Then, we have that
\begin{equation}
\operatorname{ess\;sup}\displaylimits_{t\in[0,T]}\mathbb{E}\|X_{t}\|_{H}^{2}+2\E\int_{0}^{T}\int_{\T^{d}}\psi^{\lambda}(a\nabla X_{s})\,d\xi\,ds+2\varepsilon\mathbb{E}\int_{0}^{T}\Acal(X_{s},X_{s})\,ds\le\E\|x_{n}\|_{H}^{2}.\label{eq:energy-bound}
\end{equation}
\end{prop}

\begin{proof}
Let $x_{n}\in S$, $m\in\mathbb{N}$. We may apply the It\^{o} formula
\cite[Theorem 4.2.5]{PR07-2} for the Gelfand triple $S\subset H\subset S^{\ast}$
and the process \eqref{eq:approx-eq-B1B2B3}. Let $A^{\lambda,\delta}:S\to S^{\ast}$
be defined by
\[
A^{\lambda,\delta}(x):=-\operatorname{div}(a^{\ast}\phi^{\lambda}(a\nabla x))-\frac{1}{2}J_{\delta}^{a}L^{b}J_{\delta}^{a}x,\quad x\in S.
\]
We get for $t\in[0,T]$, after taking the expected value,
\[
\begin{aligned} & \mathbb{E}\|X_{t}^{m}\|_{H}^{2}-\E\|x_{n}\|_{H}^{2}+2\varepsilon\mathbb{E}\int_{0}^{t}\Acal(X_{s}^{m},X_{s}^{m})\,ds\\
\le & \mathbb{E}\int_{0}^{t}\left[-2\brrr{A^{\lambda,\delta}(X_{s}^{m}),X_{s}^{m}}+\|B^{\delta,m}(X_{s}^{m})\|_{L_{2}(U,H)}^{2}\right]\,ds\\
\le & \mathbb{E}\int_{0}^{t}\left[-2\brrr{A^{\lambda,\delta}(X_{s}^{m}),X_{s}^{m}}+\left(\|B^{\delta,m}(X_{s}^{m})-B^{\delta,m}(X_{s})\|_{L_{2}(U,H)}+\|B^{\delta,m}(X_{s})\|_{L_{2}(U,H)}\right)^{2}\right]\,ds\\
\le & \mathbb{E}\int_{0}^{t}\left[-2\brrr{A^{\lambda,\delta}(X_{s}^{m}),X_{s}^{m}}+\left(C\|X_{s}^{m}-X_{s}\|_{H}+\|B^{\delta,m}(X_{s})\|_{L_{2}(U,H)}\right)^{2}\right]\,ds,
\end{aligned}
\]
where $X$ is as in \eqref{eq:approx-eq-B1B2} and $C$ does not depend
on $m$. By \cite[Theorem 4.6]{GessToelle14} (cf. Theorem \ref{thm:ex-limit-sln}),
$X^{m}\to X$ in $L^{2}(\Omega;C([0,T];H))$ as $m\to\infty$. Note
that $A^{\lambda,\delta}:S\to S^{\ast}$ is monotone, single-valued,
continuous and bounded and thus by Minty's trick (see e.g. \cite[Remark 4.1.1]{PR07-2}),
$A^{\lambda,\delta}(X^{m})\rightharpoonup A^{\lambda,\delta}(X)$
weakly in $S^{\ast}$ as $m\to\infty$. Hence by (i), (iii) above,
by lower semi-continuity of $\Acal$ (see \cite[Chapter I, Lemma 2.12]{MR})
and by Lebesgue's dominated convergence theorem, we converge to the
inequality
\[
\begin{aligned} & \mathbb{E}\|X_{t}\|_{H}^{2}-\E\|x_{n}\|_{H}^{2}+2\varepsilon\mathbb{E}\int_{0}^{t}\Acal(X_{s},X_{s})\,ds\\
\le & \mathbb{E}\int_{0}^{t}\left[-2\brrr{A^{\lambda,\delta}(X_{s}),X_{s}}+\|B^{\delta}(X_{s})\|_{L_{2}(U,H)}^{2}\right]\,ds.
\end{aligned}
\]
Note that, as above, for $y\in H$, we have that 
\[
\|B^{\delta}(y)\|_{L_{2}(U,H)}^{2}=\Bcal(J_{\delta}^{a}y,J_{\delta}^{a}y).
\]
On the other hand, 
\[
(J_{\delta}^{a}L^{b}J_{\delta}^{a}y,y)_{H}=-\Bcal(J_{\delta}^{a}y,J_{\delta}^{a}y),
\]
for all $y\in H$. We get that
\[
\begin{aligned} & \mathbb{E}\|X_{t}\|_{H}^{2}-\E\|x_{n}\|_{H}^{2}+2\varepsilon\mathbb{E}\int_{0}^{t}\Acal(X_{s},X_{s})\,ds\\
\le & -2\E\int_{0}^{t}(\phi^{\lambda}(a\nabla X_{s}),a\nabla X_{s})_{H}\,ds\\
\le & -2\E\int_{0}^{t}\int_{\T^{d}}\psi^{\lambda}(a\nabla X_{s})\,d\xi\,ds,
\end{aligned}
\]
which yields the claim.

\end{proof}
Let $\beta>0$. Define renormed spaces $H_{\beta}:=L^{2}(\mathbb{T}^{d})$
with norm $\|u\|_{\beta}^{2}:=\Acal_{1-2K}^{(\beta)}(u,u)$. Obviously,
$\|\cdot\|_{H}\le\|\cdot\|_{\beta}\le\sqrt{2\beta+1}\|\cdot\|_{H}$
and $S\hookrightarrow H_{\beta}\hookrightarrow S^{\ast}$ forms a
family of Gelfand triples. The following theorem is a modification
of \cite[Theorem 15]{Toelle2018proceedings}.
\begin{thm}
\label{thm:main-bound} Suppose that conditions (N), (E), (R), and
(D) hold, and that condition (BE) holds for $a$ with $K\le0$. Let
$\lambda,\delta,\varepsilon>0$. Let $x_{n}\in L^{2}(\Omega,\Fcal_{0},\P;S)$,
and let $X=X^{n,\lambda,\delta,\varepsilon}$ be a limit solution
to \eqref{eq:approx-eq-B1B2}. Then, we have that
\begin{equation}
\operatorname{ess\;sup}\displaylimits_{t\in[0,T]}\mathbb{E}\left[\Acal_{1-2K}(X_{t},X_{t})\right]+2\varepsilon\E\int_{0}^{T}\|L^{a}X_{t}\|_{H}^{2}\,dt\le e^{-(4\eps(1-2K)+c)T}\E\left[\Acal_{1-2K}(x_{n},x_{n})\right],\label{eq:a_priori}
\end{equation}
where $c\in\mathbb{R}$ is as in \eqref{eq:WDC}.
\end{thm}

\begin{proof}
We shall apply the It\^{o} formula \cite[Theorem 4.2.5]{PR07-2} for
the Gelfand triple $S\subset H_{\beta}\subset S^{\ast}$ and the process
\eqref{eq:approx-eq-B1B2B3}. As in the proof of Proposition \ref{prop:H_Bound},
we get for $t\in[0,T]$, after taking the expected value,
\[
\begin{aligned} & \mathbb{E}\|X_{t}^{m}\|_{\beta}^{2}-\mathbb{\E}\|x_{n}\|_{\beta}^{2}\\
\le & \mathbb{E}\int_{0}^{t}\left[-2\brrr{A^{\lambda,\delta,\varepsilon}(X_{s}^{m}),\tilde{L}{}^{(1/\beta)}X_{s}^{m}}+\left(C\|X_{s}^{m}-X_{s}\|_{\beta}+\|B^{\delta,m}(X_{s})\|_{L_{2}(U,H_{\beta})}\right)^{2}\right]\,ds,
\end{aligned}
\]
where $X$ is as in \eqref{eq:approx-eq-B1B2} and $C$ does not depend
on $m$. Also, note that due to the renorming of $S$, we set here
$\tilde{L}:=L^{a}+2K-1$ and
\[
\tilde{L}^{(\mu)}:=(L^{a}+2K-1)^{(\mu)}=\frac{1}{\mu}((1-\mu(L^{a}+2K-1))^{-1}-1).
\]
Now, we argue essentially as in the proof of Proposition \ref{prop:H_Bound}
and obtain that
\[
\begin{aligned} & \mathbb{E}\|X_{t}\|_{\beta}^{2}-\mathbb{\E}\|x_{n}\|_{\beta}^{2}\\
\le & \mathbb{E}\int_{0}^{t}\left[-2\brrr{A^{\lambda,\delta,\varepsilon}(X_{s}),\tilde{L}{}^{(1/\beta)}X_{s}}+\|B^{\delta}(X_{s})\|_{L_{2}(U,H_{\beta})}^{2}\right]\,ds,
\end{aligned}
\]
where we have used the commutation of $\tilde{L}{}^{(1/\beta)}$ and
$L^{a}+2K-1$ in $H$. Note that by Lemma \ref{lem:resolvent-contraction-lemma}
and Remark \ref{rem:Note-that} (ii), denoting $\tilde{J}_{\delta}:=(1-\delta(L^{a}+2K-1))^{-1}$,
for every $y\in S$,
\begin{equation}
\begin{aligned} & \brrr{\operatorname{div}(a^{\ast}\phi^{\lambda}(a\nabla y)),\tilde{L}{}^{(1/\beta)}y}\\
= & \brrr{\operatorname{div}(a^{\ast}\phi^{\lambda}(a\nabla y)),\beta\tilde{J}_{1/\beta}y-\beta y)_{H}}\\
\le & \beta\left[\int_{\T^{d}}\psi_{\lambda}(a\nabla\tilde{J}_{1/\beta}y)\,d\xi-\int_{\T^{d}}\psi_{\lambda}(a\nabla y)\,d\xi\right]\\
\le & 0.
\end{aligned}
\label{eq:wangtrick}
\end{equation}
An application of Lemma \ref{lem:commutator-estimate} shows that
there exists $c\le0$ such that
\[
\begin{aligned} & \mathbb{E}\|X_{t}\|_{\beta}^{2}-\mathbb{\E}\|x_{n}\|_{\beta}^{2}\\
\le & -2\varepsilon\mathbb{E}\int_{0}^{t}\brrr{\tilde{L}X_{s},\tilde{L}^{(1/\beta)}X_{s}}\,ds-c\E\int_{0}^{t}\Acal(J_{\delta}^{a}X_{s},J_{\delta}^{a}X_{s})\,ds.
\end{aligned}
\]
Let $\beta\to\infty$,
\[
\begin{aligned} & \mathbb{E}\left[\Acal_{1-2K}(X_{t},X_{t})\right]-\mathbb{\E}\left[\Acal_{1-2K}(x_{n},x_{n})\right]\\
\le & -2\varepsilon\mathbb{E}\int_{0}^{t}\|L^{a}X_{s}\|^{2}\,ds+4\eps\mathbb{E}\int_{0}^{t}\Acal_{1-2K}(X_{s},X_{s})\,ds-c\E\int_{0}^{t}\Acal_{1-2K}(X_{s},X_{s})\,ds,
\end{aligned}
\]
where we use weak convergence, weak lower semi-continuity and the
Mosco convergence\footnote{See \cite{A} for the notion of \emph{Mosco convergence}.}
of $\|\cdot\|_{\beta}^{2}\overset{M}{\longrightarrow}\Acal_{1-2K}(\cdot,\cdot)$.
Compare also \cite[Chapter I, Lemma 2.12]{MR}. The claim follows
by resolvent contraction in $S$ and Gronwall's lemma.
\end{proof}

\subsection{Passage to the limit}

Let us continue with the proof of the main result Theorem \ref{thm:mainmainthm}.
We shall construct an approximation in several steps, such that the
limit is a time-continuous process which is an SVI-solution to \eqref{eq:main-spde-1}.

\subsubsection*{Singular limit $\lambda\to0$. }

Fix $\delta>0$, $\eps>0$, $n\in\N$. Let $X^{\lambda,m}=X^{m}$
be the unique continuous solution to \eqref{eq:approx-eq-B1B2B3}.
For two solutions $X^{\lambda_{1},m}$, $X^{\lambda_{2},m}$, $\lambda_{1},\lambda_{2}>0$
with initial condition $x_{n}\in L^{2}(\Omega,\Fcal_{0},\P;S)$ we
have that
\[
\begin{aligned}\E\|X_{t}^{\lambda_{1},m}-X_{t}^{\lambda_{2},m}\|_{H}^{2}= & -2\E\int_{0}^{t}\brrr{A^{\lambda_{1},\delta,\varepsilon}(X_{s}^{\lambda_{1},m})-A^{\lambda_{2},\delta,\varepsilon}(X_{s}^{\lambda_{2},m}),X_{s}^{\lambda_{1},m}-X_{s}^{\lambda_{2},m}}\,ds\\
 & +\E\int_{0}^{t}\|B^{\delta,m}(X_{s}^{\lambda_{1},m}-X_{s}^{\lambda_{2},m})\|_{L_{2}(U,H)}^{2}\,ds.
\end{aligned}
\]
Note that, by monotonicity, the drift term gives a non-positive contribution.
Let $m\to\infty$ and use monotonicity, Minty's trick and Fatou's
lemma as in the proof of Proposition \ref{prop:H_Bound} to get that
\[
\begin{aligned}\E\|X_{t}^{\lambda_{1}}-X_{t}^{\lambda_{2}}\|_{H}^{2}= & -2\E\int_{0}^{t}\brrr{A^{\lambda_{1},\delta,\varepsilon}(X_{s}^{\lambda_{1}})-A^{\lambda_{2},\delta,\varepsilon}(X_{s}^{\lambda_{2}}),X_{s}^{\lambda_{1}}-X_{s}^{\lambda_{2}}}\,ds\\
 & +\E\int_{0}^{t}\|B^{\delta}(X_{s}^{\lambda_{1}}-X_{s}^{\lambda_{2}})\|_{L_{2}(U,H)}^{2}\,ds.
\end{aligned}
\]
Note that by (N) and \cite[equation (A.6) in the appendix]{GessToelle15},
we have that
\[
\langle\phi^{\lambda_{1}}(z_{1})-\phi^{\lambda_{2}}(z_{2}),z_{1}-z_{2}\rangle\ge-C(\lambda_{1}+\lambda_{2})(1+|z_{1}|^{2}+|z_{2}|^{2})\quad\forall z_{1},z_{2}\in\R^{d}.
\]
We get that
\[
\begin{aligned}\E\|X_{t}^{\lambda_{1}}-X_{t}^{\lambda_{2}}\|_{H}^{2}\le & C(\lambda_{1}+\lambda_{2})\E\int_{0}^{t}(1+\Acal(X_{s}^{\lambda_{1}},X_{s}^{\lambda_{1}})+\Acal(X_{s}^{\lambda_{2}},X_{s}^{\lambda_{2}}))\,ds.\end{aligned}
\]
Hence by \eqref{eq:energy-bound} and \eqref{eq:a_priori},
\[
\operatorname{ess\;sup}\displaylimits_{t\in[0,T]}\E\|X_{t}^{\lambda_{1}}-X_{t}^{\lambda_{2}}\|_{H}^{2}\le C(T)(\lambda_{1}+\lambda_{2})\left(\E\left[\Acal_{1-2K}(x_{n},x_{n})\right]+1\right).
\]
Hence there exists an $\{\Fcal_{t}\}$-adapted process $X\in C([0,T];L^{2}(\Omega;H))$
with $X_{0}=x_{n}$ such that $X^{\lambda}\to X$ strongly in $L^{\infty}([0,T];L^{2}(\Omega;H))$
as $\lambda\to0$.

\subsubsection*{Spatial rough limit $\delta\to0$.}

Fix $\lambda\ge0,\eps>0,n\in\N$, with $\psi^{0}:=\psi$ and $\phi^{0}:=\phi$.
Let $X^{\delta,m}=X^{m}$ be the unique continuous solution to \eqref{eq:approx-eq-B1B2B3}.
For two solutions $X^{\delta_{1},m}$, $X^{\delta_{2},m}$, $\delta_{1},\delta_{2}>0$
with initial condition $x_{n}\in L^{2}(\Omega,\Fcal_{0},\P;S)$ we
have that
\[
\begin{aligned}\E\|X_{t}^{\delta_{1},m}-X_{t}^{\delta_{2},m}\|_{H}^{2}= & -2\E\int_{0}^{t}\brrr{\eta_{s}^{\lambda,\delta_{1},\eps,m}-\eta_{s}^{\lambda,\delta_{2},\eps,m},X_{s}^{\delta_{1},m}-X_{s}^{\delta_{2},m}}\,ds\\
 & +\E\int_{0}^{t}\|B^{\delta_{1},m}(X_{s}^{\delta_{1},m})-B^{\delta_{2},m}(X_{s}^{\delta_{2},m})\|_{L_{2}(U,H)}^{2}\,ds,
\end{aligned}
\]
for some process $\eta^{\lambda,\delta,\eps,m}\in A^{\lambda,\delta,\varepsilon}(X^{\delta,m})$
$\P\otimes ds$-a.e., where the multivalued situation can only occur
when $\lambda=0$. Note that, again by monotonicity, the drift term
gives a non-positive contribution. Hence we let $m\to\infty$ and
use Fatou's lemma and Minty's trick as in the proof of Proposition
\ref{prop:H_Bound} to get that
\[
\begin{aligned} & \E\|X_{t}^{\delta_{1}}-X_{t}^{\delta_{2}}\|_{H}^{2}\\
\le & -\E\int_{0}^{t}\brrr{J_{\delta_{1}}^{a}L^{b}J_{\delta_{1}}^{a}X_{s}^{\delta_{1}}-J_{\delta_{2}}^{a}L^{b}J_{\delta_{2}}^{a}X_{s}^{\delta_{2}},X_{s}^{\delta_{1}}-X_{s}^{\delta_{2}}}\,ds\\
 & +\E\int_{0}^{t}\|B^{\delta_{1}}(X_{s}^{\delta_{1}})-B^{\delta_{2}}(X_{s}^{\delta_{2}})\|_{L_{2}(U,H)}^{2}\,ds\\
= & \E\int_{0}^{t}\left[\Bcal(J_{\delta_{1}}^{a}X_{s}^{\delta_{1}},J_{\delta_{1}}^{a}X_{s}^{\delta_{2}})+\Bcal(J_{\delta_{2}}^{a}X_{s}^{\delta_{1}},J_{\delta_{2}}^{a}X_{s}^{\delta_{2}})\right]\,ds\\
 & -2\E\int_{0}^{t}\Bcal(J_{\delta_{1}}^{a}X_{s}^{\delta_{1}},J_{\delta_{2}}^{a}X_{s}^{\delta_{2}})\,ds\\
\le & \|b\|_{\infty}^{2}\E\int_{0}^{t}\|J_{\delta_{1}}^{a}X_{s}^{\delta_{1}}\|_{S}\|(J_{\delta_{1}}^{a}-J_{\delta_{2}}^{a})X_{s}^{\delta_{1}}\|_{S}+\|J_{\delta_{2}}^{a}X_{s}^{\delta_{2}}\|_{S}\|(J_{\delta_{1}}^{a}-J_{\delta_{2}}^{a})X_{s}^{\delta_{2}}\|_{S}\,ds\\
\le & \|b\|_{\infty}^{2}\kappa^{-1}\E\int_{0}^{t}\bigg[\Acal_{1-2K}(X_{s}^{\delta_{1}},X_{s}^{\delta_{1}})^{1/2}\Acal_{1-2K}((J_{\delta_{1}}^{a}-J_{\delta_{2}}^{a})X_{s}^{\delta_{1}},(J_{\delta_{1}}^{a}-J_{\delta_{2}}^{a})X_{s}^{\delta_{1}})^{1/2}\\
 & \qquad\qquad\qquad\qquad+\Acal_{1-2K}(X_{s}^{\delta_{2}},X_{s}^{\delta_{2}})^{1/2}\Acal_{1-2K}((J_{\delta_{1}}^{a}-J_{\delta_{2}}^{a})X_{s}^{\delta_{2}},(J_{\delta_{1}}^{a}-J_{\delta_{2}}^{a})X_{s}^{\delta_{2}})^{1/2}\bigg]\,ds.
\end{aligned}
\]
Note that by the resolvent equation $J_{\delta_{1}}^{a}-J_{\delta_{2}}^{a}=(\delta_{1}-\delta_{2})J_{\delta_{1}}^{a}J_{\delta_{2}}^{a}L^{a}$,
and thus
\[
\begin{aligned} & \Acal((J_{\delta_{1}}^{a}-J_{\delta_{2}}^{a})u,(J_{\delta_{1}}^{a}-J_{\delta_{2}}^{a})u)^{1/2}\\
\le & \left[\|L^{a}J_{\delta_{1}}^{a}u\|_{H}+\|L^{a}J_{\delta_{2}}^{a}u\|_{H}\right]^{1/2}\|(J_{\delta_{1}}^{a}-J_{\delta_{2}}^{a})u\|_{H}^{1/2}\\
\le & |\delta_{1}-\delta_{2}|^{1/2}\left(\|L^{a}J_{\delta_{1}}^{a}u\|_{H}^{1/2}+\|L^{a}J_{\delta_{2}}^{a}u\|_{H}^{1/2}\right)\|L^{a}J_{\delta_{1}}^{a}J_{\delta_{2}}^{a}u\|_{H}^{1/2}
\end{aligned}
\]
According to \eqref{eq:energy-bound} and \eqref{eq:a_priori}, we
thus get an estimate
\[
\begin{aligned}\operatorname{ess\;sup}\displaylimits_{t\in[0,T]}\E\|X_{t}^{\delta_{1}}-X_{t}^{\delta_{2}}\|_{H}^{2}\le & \frac{1}{\eps}|\delta_{1}-\delta_{2}|^{1/2}C(T)\left[\E[\Acal_{1-2K}(x_{n},x_{n})]+\E\|x_{n}\|_{H}^{2}\right].\end{aligned}
\]
Thus, as $\eps>0$ is fixed, there exists an $\{\Fcal_{t}\}$-adapted
process $X=X^{\lambda,\eps,n}\in C([0,T];L^{2}(\Omega;H))$ with $X_{0}=x_{n}$
such that $X^{\delta}\to X$ strongly in $L^{\infty}([0,T];L^{2}(\Omega;H))$
as $\delta\to0$.

\subsubsection*{Vanishing viscosity limit $\eps\to0$.}

In a similar manner as before, for fixed $\lambda>0,$ $\delta>0$,
$m\in\N$, we get for two solutions $X^{\eps_{1}}$, $X^{\eps_{2}}$,
$\eps_{1},\eps_{2}>0$ with initial conditions $x_{1},x_{2}\in L^{2}(\Omega,\Fcal_{0},\P;S)$
by monotonicity arguments (after passing to $m\to\infty)$ that for
$t\in[0,T]$,
\[
\begin{aligned} & \E\|X_{t}^{\eps_{1}}-X_{t}^{\eps_{2}}\|_{H}^{2}\\
\le & \E\|x_{1}-x_{2}\|_{H}^{2}+C(\eps_{1}+\eps_{2})\E\int_{0}^{t}(\Acal(X_{s}^{\eps_{1}},X_{s}^{\eps_{1}})+\Acal(X_{s}^{\eps_{2}},X_{s}^{\eps_{2}}))\,ds.
\end{aligned}
\]
Bounding the r.h.s. by \eqref{eq:a_priori} and then taking the limits
$\lambda\to0$, $\delta\to0$, we get by the previous convergence
steps that
\begin{equation}
\begin{aligned} & \operatorname{ess\;sup}\displaylimits_{t\in[0,T]}\E\|X_{t}^{\eps_{1}}-X_{t}^{\eps_{2}}\|_{H}^{2}\\
\le & \E\|x_{1}-x_{2}\|_{H}^{2}+C(T)(\eps_{1}+\eps_{2})\left(\E\left[\Acal_{1-2K}(x_{1},x_{1})\right]+\E\left[\Acal_{1-2K}(x_{2},x_{2})\right]+1\right)
\end{aligned}
.\label{eq:eps-eps}
\end{equation}
Hence for the initial condition $x_{n}\in L^{2}(\Omega,\Fcal_{0},\P;S)$
there exists an $\{\Fcal_{t}\}$-adapted process $X\in C([0,T];L^{2}(\Omega;H))$
with $X_{0}=x_{n}$ such that $X^{\eps}\to X$ strongly in $L^{\infty}([0,T];L^{2}(\Omega;H))$
as $\eps\to0$.

\subsubsection*{Approximating the initial condition $n\to\infty$.}

Let $X^{\eps,n}$ and $X^{\eps,k}$ be the limits as above in the
``vanishing viscosity limit'' with initial conditions $x_{n},x_{k}\in L^{2}(\Omega,\Fcal_{0},\P;S)$
respectively. Taking the limit $\eps\to0$ in \eqref{eq:eps-eps},
yields
\[
\operatorname{ess\;sup}\displaylimits_{t\in[0,T]}\E\|X_{t}^{n}-X_{t}^{k}\|_{H}^{2}\le\E\|x_{n}-x_{k}\|_{H}^{2}.
\]
Hence for every initial condition $x\in L^{2}(\Omega,\Fcal_{0},\P;H)$
there exists an $\{\Fcal_{t}\}$-adapted process $X\in C([0,T];L^{2}(\Omega;H))$
with $X_{0}=x$ such that $X^{n}\to X$ strongly in $L^{\infty}([0,T];L^{2}(\Omega;H))$
as $x_{n}\to x$ and $n\to\infty$. 

\subsection{Existence of solutions}

Let $(Z_{0},Z,G,P)$ be some choice of admissible test-elements as
in Definition \ref{def:svi-new} and let $X=X^{m,n,\lambda,\delta,\eps}$
be the unique continuous solution to \eqref{eq:approx-eq-B1B2B3}
with initial condition $x_{n}\in L^{2}(\Omega,\Fcal_{0},\P;S)$. By
It\^{o}'s formula (compare with Lemma \ref{lem:stronglemma} and \cite[Step 6 of the proof of Theorem 3.1]{GessToelle15}),
\[
\begin{aligned} & \E\|X_{t}-Z_{t}\|_{H}^{2}\\
= & \E\|x_{n}-Z_{0}\|_{H}^{2}+2\E\int_{0}^{t}(\div(a^{\ast}\phi^{\lambda}(a\nabla X_{s}))+\eps L^{a}X_{s}-G_{s},X_{s}-Z_{s})_{H}\,ds\\
 & +\E\int_{0}^{t}\|B^{\delta,m}(X_{s})-B(PZ_{s})\|_{L_{2}(U,H)}^{2}\,ds+\E\int_{0}^{t}(J_{\delta}^{a}L^{b}J_{\delta}^{a}X_{s}-P^{\ast}L^{b}PZ_{s},X_{s}-Z_{s})_{H}\,ds.
\end{aligned}
\]
Recall that by Lemma \ref{lem:(N)-lemma},
\[
|\psi(\zeta)-\psi^{\lambda}(\zeta)|\le C\lambda(1+\psi(\zeta))\quad\forall\zeta\in\R^{d}.
\]
We have by integration by parts and the subgradient property of $\phi^{\lambda}$
that $\P\otimes ds$-a.e.,
\[
\begin{aligned} & (\div(a^{\ast}\phi^{\lambda}(a\nabla X)),X-Z)_{H}\\
= & (\phi^{\lambda}(a\nabla X),a\nabla(Z-X))_{H}\\
\le & \int_{\T^{d}}\psi^{\lambda}(a\nabla Z)\,d\xi-\int_{\T^{d}}\psi^{\lambda}(a\nabla X)\,d\xi\\
\le & \int_{\T^{d}}\psi(a\nabla Z)\,d\xi+\int_{\T^{d}}[\psi(a\nabla X)-\psi^{\lambda}(a\nabla X)]\,d\xi-\int_{\T^{d}}\psi(a\nabla X)\,d\xi\\
\le & \int_{\T^{d}}\psi(a\nabla Z)\,d\xi+C\lambda\int_{\T^{d}}(1+\psi(a\nabla X))\,d\xi-\int_{\T^{d}}\psi(a\nabla X)\,d\xi,
\end{aligned}
\]
recalling that $\psi^{\lambda}\le\psi$ for every $\lambda>0$. Now,
Young inequality implies that,
\[
2(\eps^{2/3}L^{a}X,\eps^{1/3}(X-Z))_{H}\le\eps^{4/3}\|L^{a}X\|_{H}^{2}+\eps^{2/3}\|X-Z\|_{H}^{2}.
\]
Hence,
\[
\begin{aligned} & \E\|X_{t}-Z_{t}\|_{H}^{2}+2\E\int_{0}^{t}\int_{\T^{d}}\psi(a\nabla X_{s})\,d\xi ds\\
\le & \E\|x_{n}-Z_{0}\|_{H}^{2}+2\E\int_{0}^{t}\int_{\T^{d}}\psi(a\nabla Z_{s})\,d\xi ds+2C\lambda\E\int_{0}^{t}\int_{\T^{d}}(1+\psi(a\nabla X_{s}))\,d\xi ds\\
 & -2\E\int_{0}^{t}(G_{s},X_{s}-Z_{s})_{H}\,ds\\
 & +2\E\int_{0}^{t}\left(\eps^{4/3}\|L^{a}X_{s}\|_{H}^{2}+\eps^{2/3}\|X_{s}-Z_{s}\|_{H}^{2}\right)\,ds\\
 & +\E\int_{0}^{t}\|B^{\delta,m}(X_{s})-B(PZ_{s})\|_{L_{2}(U,H)}^{2}\,ds+\E\int_{0}^{t}(J_{\delta}^{a}L^{b}J_{\delta}^{a}X_{s}-P^{\ast}L^{b}PZ_{s},X_{s}-Z_{s})_{H}\,ds.
\end{aligned}
\]
Now, we can pass to the limit $m\to\infty$ by convergence and lower
semi-continuity on the l.h.s., whereas on the r.h.s., we shall use
Lebesgue's dominated convergence theorem and the energy bound of Theorem
\ref{thm:main-bound}, that is, inequality \eqref{eq:a_priori}. We
get that
\[
\begin{aligned} & \E\|X_{t}-Z_{t}\|_{H}^{2}+2\E\int_{0}^{t}\int_{\T^{d}}\psi(a\nabla X_{s})\,d\xi ds\\
\le & \E\|x_{n}-Z_{0}\|_{H}^{2}+2\E\int_{0}^{t}\int_{\T^{d}}\psi(a\nabla Z_{s})\,d\xi ds+2C\lambda\E\int_{0}^{t}\int_{\T^{d}}(1+\psi(a\nabla X_{s}))\,d\xi ds\\
 & -2\E\int_{0}^{t}(G_{s},X_{s}-Z_{s})_{H}\,ds\\
 & +2\E\int_{0}^{t}\left(\eps^{4/3}\|L^{a}X_{s}\|_{H}^{2}+\eps^{2/3}\|X_{s}-Z_{s}\|_{H}^{2}\right)\,ds\\
 & +\E\int_{0}^{t}\|B(J_{\delta}^{a}X_{s}-PZ_{s})\|_{L_{2}(U,H)}^{2}\,ds+\E\int_{0}^{t}(J_{\delta}^{a}L^{b}J_{\delta}^{a}X_{s}-P^{\ast}L^{b}PZ_{s},X_{s}-Z_{s})_{H}\,ds.
\end{aligned}
\]
Reordering terms, we get after some cancellation,
\[
\begin{aligned} & \E\|X_{t}-Z_{t}\|_{H}^{2}+2\E\int_{0}^{t}\int_{\T^{d}}\psi(a\nabla X_{s})\,d\xi ds\\
\le & \E\|x_{n}-Z_{0}\|_{H}^{2}+2\E\int_{0}^{t}\int_{\T^{d}}\psi(a\nabla Z_{s})\,d\xi ds+2C\lambda\E\int_{0}^{t}\int_{\T^{d}}(1+\psi(a\nabla X_{s}))\,d\xi ds\\
 & -2\E\int_{0}^{t}(G_{s},X_{s}-Z_{s})_{H}\,ds\\
 & +2\E\int_{0}^{t}\left(\eps^{4/3}\|L^{a}X_{s}\|_{H}^{2}+\eps^{2/3}\|X_{s}-Z_{s}\|_{H}^{2}\right)\,ds\\
 & -\E\int_{0}^{t}(L^{b}PZ_{s},PX_{s}-J_{\delta}^{a}X_{s})_{H}\,ds-\E\int_{0}^{t}(J_{\delta}^{a}X_{s},L^{b}(J_{\delta}^{a}Z_{s}-PZ_{s}))_{H}\,ds.
\end{aligned}
\]
Now, first let $\lambda\to0$ and then $\delta\to0$, where we use
the bound \eqref{eq:a_priori} and Lemma \ref{lem:comm-implies-M},
taking into account that $Z\in L^{2}([0,T]\times\Omega;D(L^{b})\cap S)$
and that $X^{\delta}\to X$ converges strongly in $L^{\infty}([0,T];L^{2}(\Omega;H))$,
noting that $\Bcal(u,u)\le\|b\|_{\infty}^{2}\kappa^{-1}\Acal(u,u)$,
$u\in S$, and thus Lebesgue's dominated convergence theorem can be
applied in order to get that
\[
\begin{aligned} & \E\|X_{t}-Z_{t}\|_{H}^{2}+2\E\int_{0}^{t}\int_{\T^{d}}\psi(a\nabla X_{s})\,d\xi ds\\
\le & \E\|x_{n}-Z_{0}\|_{H}^{2}+2\E\int_{0}^{t}\int_{\T^{d}}\psi(a\nabla Z_{s})\,d\xi ds\\
 & -2\E\int_{0}^{t}(G_{s},X_{s}-Z_{s})_{H}\,ds\\
 & +2\E\int_{0}^{t}\left(\eps^{4/3}\|L^{a}X_{s}\|_{H}^{2}+\eps^{2/3}\|X_{s}-Z_{s}\|_{H}^{2}\right)\,ds\\
 & -\E\int_{0}^{t}(L^{b}PZ_{s},PX_{s}-X_{s})_{H}\,ds-\E\int_{0}^{t}(X_{s},L^{b}(Z_{s}-PZ_{s}))_{H}\,ds.
\end{aligned}
\]
Finally, we can let $\eps\to0$ and use the bounds \eqref{eq:energy-bound}
and \eqref{eq:a_priori}, so that we get together with the lower semi-continuity
of $\Psi$ that
\[
\begin{aligned} & \E\|X_{t}-Z_{t}\|_{H}^{2}+2\E\int_{0}^{t}\Psi(X_{s})ds\\
\le & \E\|x_{n}-Z_{0}\|_{H}^{2}+2\E\int_{0}^{t}\int_{\T^{d}}\psi(a\nabla Z_{s})\,d\xi ds\\
 & -2\E\int_{0}^{t}(G_{s},X_{s}-Z_{s})_{H}\,ds\\
 & -\E\int_{0}^{t}(L^{b}PZ_{s},PX_{s}-X_{s})_{H}\,ds-\E\int_{0}^{t}(X_{s},L^{b}(Z_{s}-PZ_{s}))_{H}\,ds.
\end{aligned}
\]
Note that, as in particular $Z\in L^{2}([0,T]\times\Omega;S)$,
\[
\int_{\T^{d}}\psi(a\nabla Z)\,d\xi=\Psi(Z).
\]
Passing to $n\to\infty$ and using convergence and lower semi-continuity
and the bound \eqref{eq:energy-bound} again, yields the existence
of a time-continuous and adapted SVI-solution for equation \eqref{eq:main-spde-1}.
The regularity of SVI solutions \eqref{eq:SVI_regularity-Stratonovich}
follows from passing to the limit in equation \eqref{eq:energy-bound}.
The existence part of Theorem \ref{thm:mainmainthm} is proved.

\subsection{Uniqueness}

Compare with \cite{GessToelle15,Gess:2016kda,Gess:2015gw,Ciotir:2016fe,VBMR}.
Let $X\in L^{2}([0,T]\times\Omega;H)$ be any SVI solution to \eqref{eq:main-spde-1}
with initial datum $x\in L^{2}(\Omega,\Fcal_{0},\P;H)$. Let $Z_{0}=y_{n}$,
$Z=Z^{\lambda,\delta,\eps,n}=X^{\lambda,\delta,\eps,n}$, the strong
(!) solution to \eqref{eq:approx-eq-B1B2} with initial datum $y_{n}\in L^{2}(\Omega,\Fcal_{0},\P;S)$.
Let $G=G^{\lambda,\eps,n}=\operatorname{div}(a^{\ast}\phi^{\lambda}(a\nabla Z))+\varepsilon L^{a}Z$
and $P=P^{\delta}=J_{\delta}^{a}$. Again, we omit the indices, whenever
it seems convenient. By the energy estimates \eqref{eq:energy-bound}
and \eqref{eq:a_priori}, the integrals are finite\footnote{We would like to point out, since $\lambda,\delta,\eps>0$, and $y_{n}\in L^{2}(\Omega,\Fcal_{0},\P;S)$,
the conditions of \cite[Theorem 4.2.5]{PR07-2} are satisfied and
we have a pathwise single-valued solution which satisfies the It\^{o}-equation
\eqref{eq:approx-eq-B1B2} in the Gelfand triple $S\hookrightarrow H\hookrightarrow S^{\ast}$.
However, by our a priori estimates \eqref{eq:energy-bound} and \eqref{eq:a_priori},
we get additional regularity and $(y_{n},X^{\lambda,\delta,\eps,n},G^{\lambda,\eps,n},J_{\delta}^{a})$
is indeed an quadruple of admissible test-elements.}. By the definition of SVI-solutions, we get for $t\in[0,T]$ that
\[
\begin{aligned} & \E\|X_{t}-Z_{t}\|_{H}^{2}+2\E\int_{0}^{t}\Psi(X_{s})ds\\
\le & \E\|x-y_{n}\|_{H}^{2}+2\E\int_{0}^{t}\Psi(Z_{s})ds\\
 & -2\E\int_{0}^{t}(\operatorname{div}(a^{\ast}\phi^{\lambda}(a\nabla Z_{s}))+\varepsilon L^{a}Z_{s},X_{s}-Z_{s})_{H}\,ds\\
 & -\E\int_{0}^{t}(L^{b}J_{\delta}^{a}Z_{s},J_{\delta}^{a}X_{s}-X_{s})_{H}\,ds-\E\int_{0}^{t}(X_{s},L^{b}(Z_{s}-J_{\delta}^{a}Z_{s}))_{H}\,ds.
\end{aligned}
\]
By Lemma \ref{lem:(N)-lemma}, for all $w\in S$, we have 
\[
-(\operatorname{div}(a^{\ast}\phi^{\lambda}(a\nabla Z)),w-Z)_{H}+\Psi(Z)\le\Psi(w)+C\lambda(1+\Psi(Z))\quad\P\otimes ds-\text{a.e}.
\]
Since $\Psi$ is the lower semi-continuous envelope of $\tilde{{\Psi}}=\Psi\vert_{S}$
(i.e., $\Psi$ restricted to $S$), for a.e. $(t,\omega)\in[0,T]\times\Omega$,
we can choose a sequence $w^{k}\in S$, $k\in\N$ such that $w^{k}\to X_{s}(\omega)$
in $H$ and $\Psi(w^{k})\to\Psi(X_{t}(\omega))$.

Hence, 
\[
-(\operatorname{div}(a^{\ast}\phi^{\lambda}(a\nabla Z)),X-Z)_{H}+\Psi(Z)\le\Psi(X)+C\lambda(1+\Psi(Z))\quad\P\otimes ds-\text{a.e}.
\]
Thus,
\[
\begin{aligned} & \E\|X_{t}-Z_{t}\|_{H}^{2}\\
\le & \E\|x-y_{n}\|_{H}^{2}\\
 & +C\lambda\E\int_{0}^{t}(1+\Psi(Z_{s}))\,ds\\
 & +2\E\int_{0}^{t}\left(\eps^{4/3}\|L^{a}Z_{s}\|_{H}^{2}+\eps^{2/3}\|X_{s}-Z_{s}\|_{H}^{2}\right)\,ds\\
 & -\E\int_{0}^{t}(L^{b}J_{\delta}^{a}Z_{s},J_{\delta}^{a}X_{s}-X_{s})_{H}\,ds-\E\int_{0}^{t}(X_{s},L^{b}(Z_{s}-J_{\delta}^{a}Z_{s}))_{H}\,ds.
\end{aligned}
\]
We can take the limit $\lambda\to0$ by using the bound \eqref{eq:a_priori},
which is uniform in $\lambda$ and $\delta$, and get that
\[
\begin{aligned} & \E\|X_{t}-Z_{t}^{\delta}\|_{H}^{2}\\
\le & \E\|x-y_{n}\|_{H}^{2}\\
 & +2\E\int_{0}^{t}\left(\eps^{4/3}\|L^{a}Z_{s}^{\delta}\|_{H}^{2}+\eps^{2/3}\|X_{s}-Z_{s}^{\delta}\|_{H}^{2}\right)\,ds\\
 & -\E\int_{0}^{t}(L^{b}J_{\delta}^{a}Z_{s}^{\delta},J_{\delta}^{a}X_{s}-X_{s})_{H}\,ds-\E\int_{0}^{t}(X_{s},L^{b}(Z_{s}^{\delta}-J_{\delta}^{a}Z_{s}^{\delta}))_{H}\,ds.
\end{aligned}
\]
Due to the bound \eqref{eq:a_priori} and Lemma \ref{lem:comm-implies-M},
we can use the $\P\otimes ds$-a.e. strong convergence of $J_{\delta}^{a}X\to X$
in $H$ and Lebesgue's dominated convergence theorem, the bound \eqref{eq:energy-bound},
the strong convergence $Z^{\delta}\to Z$ in $L^{\infty}([0,T];L^{2}(\Omega;H))$
and the weak convergence of $Z^{\delta}\rightharpoonup Z$ in $L^{2}([0,T];L^{2}(\Omega;D(L^{b})\cap S))$
in order to let $\delta\to0$ so that the above expression converges
to
\[
\begin{aligned} & \E\|X_{t}-Z_{t}^{\eps}\|_{H}^{2}\\
\le & \E\|x-y_{n}\|_{H}^{2}\\
 & +2\E\int_{0}^{t}\left(\eps^{4/3}\|L^{a}Z_{s}^{\eps}\|_{H}^{2}+\eps^{2/3}\|X_{s}-Z_{s}^{\eps}\|_{H}^{2}\right)\,ds.
\end{aligned}
\]
Now, for $\eps\to0$, using the bound \eqref{eq:a_priori} that the
expression converges to
\[
\begin{aligned}\E\|X_{t}-Z_{t}^{n}\|_{H}^{2}\le & \E\|x-y_{n}\|_{H}^{2},\end{aligned}
\]
for a.e. $t\in[0,T]$. The bound \eqref{eq:uniqueness-bound} follows
by approximating initial data $y\in L^{2}(\Omega,\Fcal_{0},\P;H)$
by $y_{n}\to y$, i.e., a strongly convergent sequence in $L^{2}(\Omega,\Fcal_{0},\P;H)$
as $n\to\infty$, and using lower semi-continuity. The uniqueness
part of Theorem \ref{thm:mainmainthm} is proved.

\appendix

\section{\label{sec:app-ex} Existence of approximating solutions}

Let us recall the following conditions from \cite{GessToelle14},
simplified with regard to the time-dependence of the drift coefficients,
which is not needed here. Suppose that $A:S\to2^{S^{\ast}}$ satisfies
the following conditions: There is a constant $C>0$ such that
\begin{enumerate}
\item[(A1)]  The map $x\mapsto A(x)$ is maximal monotone with non-empty values.
\item[(A2)]  For all $x\in S$, for all $y\in A(x)$:
\[
\|y\|_{S^{\ast}}\le C\|x\|_{S}.
\]
\item[(A3)]  For all $x\in S$, for all $y\in A(x)$, and for all $\mu>0$:
\[
2_{S^{\ast}}\langle y,L^{(\mu)}x\rangle_{S}\le C\|x\|_{S}^{2},
\]
such that $C$ is independent of $\mu$, where $L^{(\mu)}:=\frac{1}{\mu}((1-\mu L)^{-1}-1)$
denotes the Yosida-approximation of $L=\Delta$. 
\end{enumerate}
Let $U$ be a separable Hilbert space. Denote the \emph{space of Hilbert-Schmidt
operators} from $U$ to $H$ by $L_{2}(U,H)$. Suppose that $B:[0,T]\times\Omega\times S\to L_{2}(U,H)$
is \emph{progressively measurable}\footnote{That is, for every $t\in[0,T]$ the map $B:[0,t]\times\Omega\times S\to L_{2}(U,H)$
is $\mathcal{B}([0,t])\otimes\mathcal{F}_{t}\otimes\mathcal{B}(S)$-measurable.} and that there exist constants $C_{1},C_{2},C_{3}>0$ such that
\begin{enumerate}
\item[(B1)]  There is $h\in L^{1}([0,T]\times\Omega)$ such that
\[
\|B_{t}(x)\|_{L_{2}(U,H)}^{2}\le C_{1}\|x\|_{S}^{2}+h_{t}
\]
for all $t\in[0,T]$, $x\in S$ and $\omega\in\Omega$.
\item[(B2)]  
\[
\|B_{t}(x)-B_{t}(y)\|_{L_{2}(U,H)}^{2}\le C_{2}\|x-y\|_{H}^{2}
\]
for all $t\in[0,T]$, $x,y\in S$ and $\omega\in\Omega$.
\item[(B3)]  There is $\tilde{h}\in L^{1}([0,T]\times\Omega)$ such that
\[
\|B_{t}(x)\|_{L_{2}(U,S)}^{2}\le C_{3}\|x\|_{S}^{2}+\tilde{h}_{t}
\]
for all $t\in[0,T]$, $x\in S$ and $\omega\in\Omega$.
\end{enumerate}
Denote by $\{W_{t}\}_{t\ge0}$ a cylindrical Wiener process in $U$
for the stochastic basis $(\Omega,\mathcal{F},\{\mathcal{F}_{t}\}_{t\ge0},\mathbb{P})$.

\begin{defn}
\label{def:sln-GT} We say that a continuous $\{\mathcal{F}_{t}\}_{t\ge0}$-adapted
stochastic process $X:[0,T]\times\Omega\to H$ is a \emph{solution
}to
\begin{equation}
dX_{t}+A(X_{t})\,dt\ni B_{t}(X_{t})\,dW_{t},\quad X_{0}=x,\label{eq:Ito-multivalued}
\end{equation}
if $X\in L^{2}(\Omega;C([0,T];H))\cap L^{2}([0,T]\times\Omega;S)$
and solves the following integral equation in $S^{\ast}$
\[
X_{t}=x-\int_{0}^{t}\eta_{s}\,ds+\int_{0}^{t}B_{s}(X_{s})\,dW_{s},
\]
$\mathbb{P}$-a.s. for all $t\in[0,T]$, where $\eta\in A(X)$, $dt\otimes\mathbb{P}$-a.s.
\end{defn}

\begin{thm}
\label{thm:sln-ex-GT} Suppose that conditions (A1)--(A3), (B1)--(B3)
hold. Let $x\in L^{2}(\Omega,\mathcal{F}_{0},\mathbb{P};S)$. Then
there exists a unique solution in the sense of the previous definition
to the equation
\begin{equation}
dX_{t}+A(X_{t})\,dt\ni B_{t}(X_{t})\,dW_{t},\quad X_{0}=x,\label{eq:abstract-Ito-eq-solution}
\end{equation}
that satisfies
\[
\E\left[\sup_{t\in[0,T]}\|X_{t}\|_{S}^{2}\right]<\infty.
\]
\end{thm}

\begin{proof}
See \cite[Theorem 4.4]{GessToelle14}.
\end{proof}
\begin{defn}
\label{def:limit-solution}An $\{\mathcal{F}_{t}\}_{t\ge0}$-adapted
stochastic process $X\in L^{2}(\Omega;C([0,T];H))$ is called a \emph{limit
solution }to \eqref{eq:Ito-multivalued} with starting point $x\in H$
if for all approximations $x^{m}\in S$, $m\in\mathbb{N}$ with $\|x^{m}-x\|_{H}\to0$
as $m\to\infty$ and all $B^{m}$ satisfying (B1)--(B3) and such
that $B^{m}(y)\to B(y)$ strongly in $L^{2}([0,T]\times\Omega;L_{2}(U,H))$
for every $y\in S$, we have that
\[
X^{m}\to X\quad\text{strongly in \ensuremath{L^{2}(\Omega;C([0,T];H))} as \ensuremath{m\to\infty}.}
\]
\end{defn}

\begin{thm}
\label{thm:ex-limit-sln} Suppose that conditions (A1)--(A3), (B1)--(B2)
hold. Let $x\in L^{2}(\Omega,\mathcal{F}_{0},\mathbb{P};H)$. Then
there exists a unique limit solution in the sense of the previous
definition to the equation
\begin{equation}
dX_{t}+A(X_{t})\,dt\ni B_{t}(X_{t})\,dW_{t},\quad X_{0}=x.\label{eq:abstract-Ito-eq-limit-solution}
\end{equation}
\end{thm}

\begin{proof}
See \cite[Theorem 4.6]{GessToelle14}.
\end{proof}

\section{\label{sec:Proof} Remaining proofs}
\begin{proof}[Proof of Proposition \ref{prop:BEcond}.]
 Suppose that $a\in C^{2}(\T^{d};\R^{d\times d})$ and that for all
$1\le i,j\le d$,
\begin{equation}
\sum_{k=1}^{d}\sum_{q=1}^{d}\left[a_{qj}\partial_{k}a_{qi}+a_{qi}\partial_{k}a_{qj}\right]=0\quad\text{on }\T^{d},\label{eq:BEsufficient-1}
\end{equation}
where $a=(a_{ij})$. Let $f\in C^{3}(\T^{d})$. Let $\A:=a^{\ast}a$.
Utilizing the Einstein summation convention, we get that
\[
\begin{aligned} & \frac{1}{2}L^{a}|a\nabla f|^{2}-\langle a\nabla f,a\nabla L^{a}f\rangle\\
= & \frac{1}{2}\partial_{k}(a_{pk}a_{pl}\partial_{l}(a_{qi}\partial_{i}fa_{qj}\partial_{j}f))-a_{qi}\partial_{i}fa_{qj}\partial_{j}\partial_{k}(a_{pk}a_{pl}\partial_{l}f)\\
= & \partial_{k}(a_{pk}a_{pl}a_{qi}\partial_{i}f\partial_{l}a_{qj}\partial_{j}f)+\partial_{k}(a_{pk}a_{pl}a_{qi}\partial_{i}fa_{qj}\partial_{l}\partial_{j}f)\\
 & -a_{qi}\partial_{i}fa_{qj}\partial_{j}(\partial_{k}a_{pk}a_{pl}\partial_{l}f)-a_{qi}\partial_{i}fa_{qj}\partial_{j}(a_{pk}\partial_{k}a_{pl}\partial_{l}f)-a_{qi}\partial_{i}fa_{qj}\partial_{j}(a_{pk}a_{pl}\partial_{k}\partial_{l}f)\\
= & \partial_{k}a_{pk}a_{pl}a_{qi}\partial_{i}f\partial_{l}a_{qj}\partial_{j}f+\partial_{k}a_{pk}a_{pl}a_{qi}\partial_{i}fa_{qj}\partial_{l}\partial_{j}f+a_{pk}\partial_{k}a_{pl}a_{qi}\partial_{i}f\partial_{l}a_{qj}\partial_{j}f\\
 & +a_{pk}\partial_{k}a_{pl}a_{qi}\partial_{i}fa_{qj}\partial_{l}\partial_{j}f+a_{pk}a_{pl}\partial_{k}a_{qi}\partial_{i}f\partial_{l}a_{qj}\partial_{j}f+a_{pk}a_{pl}\partial_{k}a_{qi}\partial_{i}fa_{qj}\partial_{l}\partial_{j}f\\
 & +a_{pk}a_{pl}a_{qi}\partial_{k}\partial_{i}f\partial_{l}a_{qj}\partial_{j}f+a_{pk}a_{pl}a_{qi}\partial_{k}\partial_{i}fa_{qj}\partial_{l}\partial_{j}f+a_{pk}a_{pl}a_{qi}\partial_{i}f\partial_{k}\partial_{l}a_{qj}\partial_{j}f\\
 & +a_{pk}a_{pl}a_{qi}\partial_{i}f\partial_{k}a_{qj}\partial_{l}\partial_{j}f+a_{pk}a_{pl}a_{qi}\partial_{i}f\partial_{l}a_{qj}\partial_{k}\partial_{j}f+a_{pk}a_{pl}a_{qi}\partial_{i}fa_{qj}\partial_{k}\partial_{l}\partial_{j}f\\
 & -a_{qi}\partial_{i}fa_{qj}\partial_{j}\partial_{k}a_{pk}a_{pl}\partial_{l}f-a_{qi}\partial_{i}fa_{qj}\partial_{k}a_{pk}\partial_{j}a_{pl}\partial_{l}f-a_{qi}\partial_{i}fa_{qj}\partial_{k}a_{pk}a_{pl}\partial_{j}\partial_{l}f\\
 & -a_{qi}\partial_{i}fa_{qj}\partial_{j}a_{pk}\partial_{k}a_{pl}\partial_{l}f-a_{qi}\partial_{i}fa_{qj}a_{pk}\partial_{j}\partial_{k}a_{pl}\partial_{l}f-a_{qi}\partial_{i}fa_{qj}a_{pk}\partial_{k}a_{pl}\partial_{j}\partial_{l}f\\
 & -a_{qi}\partial_{i}fa_{qj}\partial_{j}a_{pk}a_{pl}\partial_{k}\partial_{l}f-a_{qi}\partial_{i}fa_{qj}a_{pk}\partial_{j}a_{pl}\partial_{k}\partial_{l}f-a_{qi}\partial_{i}fa_{qj}a_{pk}a_{pl}\partial_{j}\partial_{k}\partial_{l}f\\
\ge & |a\nabla(a\nabla f)|^{2}+\operatorname{Tr}\left[(\A(D^{2}f))^{2}\right]-C\left(\|a\|_{\infty},\|\nabla a\|_{\infty},\|D^{2}a\|_{\infty}\right)|\nabla f|^{2}\\
\ge & -\kappa^{-1}C\left(\|a\|_{\infty},\|\nabla a\|_{\infty},\|D^{2}a\|_{\infty}\right)|a\nabla f|^{2},
\end{aligned}
\]
where we have used that $\operatorname{Tr}\left[(\A(D^{2}f))^{2}\right]\ge0$,
which can be seen as follows. Recall that for any $d\times d$-Matrix
$B$, we have that $\operatorname{Tr}(B^{2})=\sum_{i=1}^{d}\lambda_{i}^{2}$,
where $\lambda_{i}\in\C$ are the eigenvalues of $B$ (with distinct
indices assigned to repeated eigenvalues, if necessary). By (E), $\A$
is positive definite and symmetric. However, by polar decomposition,
$\A(D^{2}f)$ has the same spectrum as $\sqrt{\A}(D^{2}f)\sqrt{\A}$,
which, being symmetric, has only real eigenvalues. The proof is completed
by density of $C^{3}(\T^{d})$ in $\Lambda^{a}=H^{3}(\T^{d})$, where
$\Lambda^{a}$ is as in Definition \ref{def:BEdefi}.
\end{proof}

\def\cprime{$'$}


\begin{thebibliography}{10}

\bibitem{Ambrosio:2015gp}
L.~Ambrosio, N.~Gigli, and G.~Savar{\'e}.
\newblock {Bakry{\textendash}{\'E}mery curvature-dimension condition and
  Riemannian Ricci curvature bounds}.
\newblock {\em Ann. Probab.}, 43(1):339--404, 2015.

\bibitem{A}
H.~Attouch.
\newblock {\em {V}ariational convergence for functions and operators}.
\newblock Pitman, Boston--London--Melbourne, 1984.

\bibitem{ABM}
H.~Attouch, G.~Buttazzo, and G.~Michaille.
\newblock {\em {V}ariational analysis in {S}obolev and {B}{V} spaces:
  applications to {P}{D}{E}s and optimization}, volume~6 of {\em MPS-SIAM
  series on optimization}.
\newblock SIAM and MPS, Philadelphia, 2006.

\bibitem{Bakry:1997tf}
D.~Bakry.
\newblock {O}n {S}obolev and logarithmic {S}obolev inequalities for markov
  semigroups.
\newblock In {\em New trends in stochastic analysis (Charingworth, 1994)},
  pages 43--75. World Scientific Publishing, River Edge, N.J., 1997.

\bibitem{Bakry:2014ir}
D.~Bakry, I.~Gentil, and M.~Ledoux.
\newblock {A}nalysis and geometry of {M}arkov diffusion operators.
\newblock Springer International Publishing, Cham, 2014.

\bibitem{BBHT13}
V.~Barbu, Z.~Brze{\'z}{}niak, E.~Hausenblas, and L.~Tubaro.
\newblock Existence and convergence results for infinite dimensional nonlinear
  stochastic equations with multiplicative noise.
\newblock {\em Stochastic Process. Appl.}, 123(3):934--951, 2013.

\bibitem{Barbu:2017dr}
V.~Barbu, Z.~Brze{\'{z}}niak, and L.~Tubaro.
\newblock {Stochastic nonlinear parabolic equations with Stratonovich gradient
  noise}.
\newblock {\em Appl. Math. Optim.}, 188(3):1--17, 2017.

\bibitem{BDP06}
V.~Barbu and G.~Da~Prato.
\newblock Ergodicity for nonlinear stochastic equations in variational
  formulation.
\newblock {\em Appl. Math. Optim.}, 53(2):121--139, 2006.

\bibitem{Barbu:2012tq}
V.~Barbu and T.~Precupanu.
\newblock {\em {Convexity and Optimization in Banach Spaces}}.
\newblock Springer Science {\&} Business Media, 2012.

\bibitem{VBMR}
V.~Barbu and M.~R{\"o}{}ckner.
\newblock {S}tochastic variational inequalities and applications to the total
  variation flow perturbed by linear multiplicative noise.
\newblock {\em Arch. Rational Mech. Anal.}, 209(3):797--834, 2013.

\bibitem{BR15}
V.~Barbu and M.~R{\"o}{}ckner.
\newblock {A}n operatorial approach to stochastic partial differential
  equations driven by linear multiplicative noise.
\newblock {\em J. Eur. Math. Soc. (JEMS)}, 17(7):1789--1815, 2015.

\bibitem{BenRas}
A.~Bensoussan and A.~R{\u{a}}{}{\c{s}}{}canu.
\newblock {S}tochastic variational inequalities in infinite dimensional spaces.
\newblock {\em Numer. Funct. Anal. Optim.}, 18(1--2):19--54, 1997.

\bibitem{Breit:2014cda}
D.~Breit.
\newblock {Regularity theory for nonlinear systems of SPDEs}.
\newblock {\em Manuscripta math.}, 146(3-4):329--349, 2014.

\bibitem{Breit:2016dga}
D.~Breit and M.~Hofmanov{\'a}.
\newblock {On time regularity of stochastic evolution equations with monotone
  coefficients}.
\newblock {\em Comptes Rendus Mathematique}, 354(1):33--37, 2016.

\bibitem{BrzeZniak:1988gf}
Z.~Brze{\'{z}}niak, M.~Capi{\'{n}}ski, and F.~Flandoli.
\newblock {A convergence result for stochastic partial differential equations}.
\newblock {\em Stochastics}, 24(4):423--445, 1988.

\bibitem{CiotToe}
I.~Ciotir and J.~M. T{\"o}{}lle.
\newblock {C}onvergence of invariant measures for singular stochastic diffusion
  equations.
\newblock {\em Stochastic Process. Appl.}, 122(4):1998--2017, 2012.

\bibitem{Ciotir:2016fe}
I.~Ciotir and J.~M. T{\"o}lle.
\newblock {Nonlinear stochastic partial differential equations with singular
  diffusivity and gradient Stratonovich noise}.
\newblock {\em J. Funct. Anal.}, 271(7):1764--1792, 2016.

\bibitem{DaPrato:1982wn}
G.~Da~Prato, M.~Iannelli, and L.~Tubaro.
\newblock {An existence result for a linear abstract stochastic equation in
  Hilbert spaces}.
\newblock {\em Rendiconti del Seminario Matematico dell'Universit\`a di
  Padova}, 67:171--180, 1982.

\bibitem{DaPrato:1982ci}
G.~Da~Prato, M.~Iannelli, and L.~Tubaro.
\newblock {Some results on linear stochastic differential equations in
  {H}ilbert spaces}.
\newblock {\em Stochastics}, 6(2):105--116, 1982.

\bibitem{DareiotisGerencserGess2019}
K.~Dareiotis, M.~Gerencs\'{e}r, and B.~Gess.
\newblock {Entropy solutions for stochastic porous media equations}.
\newblock {\em J. Differential Equations}, 266(6):3732--3763, 2019.

\bibitem{DareiotisGess2018}
K.~Dareiotis and B.~Gess.
\newblock {Nonlinear diffusion equations with nonlinear gradient noise}.
\newblock {\em Preprint}, pages 1--42, 2018.
\newblock \url{https://arxiv.org/abs/1811.08356}.

\bibitem{ESvR}
A.~Es-Sarhir and M.-K. von Renesse.
\newblock {E}rgodicity of stochastic curve shortening flow in the plane.
\newblock {\em SIAM J. Math. Anal.}, 44(1):224--244, 2012.

\bibitem{ESvRS}
A.~Es-Sarhir, M.-K. von Renesse, and W.~Stannat.
\newblock {E}stimates for the ergodic measure and polynomial stability of plane
  stochastic curve shortening flow.
\newblock {\em Nonlinear Differential Equations Appl. (NoDEA)}, 19(6):663--675,
  2012.

\bibitem{FehrmanGess2019}
B.~Fehrman and B.~Gess.
\newblock {Well-posedness of nonlinear diffusion equations with nonlinear,
  conservative noise}.
\newblock {\em Archive for Rational Mechanics and Analysis}, 233(1):249--322,
  2019.

\bibitem{FOT}
M.~Fukushima, Y.~Oshima, and M.~Takeda.
\newblock {\em {D}irichlet forms and symmetric {M}arkov processes}.
\newblock de Gruyter, Berlin--New York, 1994.

\bibitem{Gess:2015gw}
B.~Gess and M.~R{\"o}ckner.
\newblock {S}ingular-degenerate multivalued stochastic fast diffusion
  equations.
\newblock {\em SIAM J. Math. Anal.}, 47(5):4058--4090, 2015.

\bibitem{Gess:2016kda}
B.~Gess and M.~R{\"o}ckner.
\newblock {Stochastic variational inequalities and regularity for degenerate
  stochastic partial differential equations}.
\newblock {\em Trans. Amer. Math. Soc.}, 369(5):3017--3045, 2017.

\bibitem{GessToelle14}
B.~Gess and J.~M. T{\"o}{}lle.
\newblock {M}ulti-valued, singular stochastic evolution inclusions.
\newblock {\em J. Math. Pures Appl.}, 101(6):789--827, 2014.

\bibitem{GT:2016he}
B.~Gess and J.~M. T{\"o}lle.
\newblock {Ergodicity and local limits for stochastic local and nonlocal
  $p$-Laplace equations}.
\newblock {\em SIAM J. Math. Anal.}, 48(6):4094--4125, 2016.

\bibitem{GessToelle15}
B.~Gess and J.~M. T{\"o}{}lle.
\newblock {Stability of solutions to stochastic partial differential
  equations}.
\newblock {\em J. Differential Equations}, 260(6):4973--5025, 2016.

\bibitem{Gigli:2013hg}
N.~Gigli, K.~Kuwada, and S.-i. Ohta.
\newblock {H}eat flow on {A}lexandrov spaces.
\newblock {\em Comm. Pure Appl. Math.}, 66(3):307--331, 2013.

\bibitem{Grigoryan:2009tb}
A.~Grigor{\textquoteright}yan.
\newblock {\em {Heat kernel and analysis on manifolds}}, volume~47 of {\em
  AMS/IP Studies in Advanced Mathematics}.
\newblock American Mathematical Society, Providence, RI; International Press,
  Boston, MA, 2009.

\bibitem{kunita1997stochastic}
H.~Kunita.
\newblock {\em {S}tochastic flows and stochastic differential equations}.
\newblock Cambridge Studies in Advanced Mathematics. Cambridge University
  Press, 1997.

\bibitem{Kurtz:1995vc}
T.~G. Kurtz, {\'E}.~Pardoux, and P.~Protter.
\newblock {Stratonovich stochastic differential equations driven by general
  semimartingales}.
\newblock {\em Annales de l'Institut Henri Poincar\'e. Probabilit\'es et
  Statistiques}, 31(2):351--377, 1995.

\bibitem{lax2007linear}
P.~D. Lax.
\newblock {\em {Linear Algebra and Its Applications}}.
\newblock Wiley, 2007.

\bibitem{Ledoux:2000vx}
M.~Ledoux.
\newblock {The geometry of Markov diffusion generators}.
\newblock {\em Annales de la facult{\'e} des sciences de Toulouse
  Math{\'e}matiques}, 9(2):305--366, 2000.

\bibitem{Liu09}
W.~Liu.
\newblock On the stochastic {$p$}-{L}aplace equation.
\newblock {\em J. Math. Anal. Appl.}, 360(2):737--751, 2009.

\bibitem{LiuToe1}
W.~Liu and J.~M. T{\"o}lle.
\newblock {E}xistence and uniqueness of invariant measures for stochastic
  evolution equations with weakly dissipative drifts.
\newblock {\em Electr. Comm. Probab.}, 16:447--457, 2011.

\bibitem{MR}
Z.-M. Ma and M.~R\"ockner.
\newblock {\em {I}ntroduction to the theory of (non-symmetric) {D}irichlet
  forms}.
\newblock Universitext. Springer-Verlag, Berlin--Heidelberg--New York, 1992.

\bibitem{Mali}
L.~Maligranda.
\newblock {S}ome remarks on {O}rlicz's interpolation theorem.
\newblock {\em Studia Math.}, 95(1):43--58, 1989.

\bibitem{Marinelli:2018ej}
C.~Marinelli and L.~Scarpa.
\newblock {Strong solutions to SPDEs with monotone drift in divergence form}.
\newblock {\em Stochastics and Partial Differential Equations: Analysis and
  Computations}, 6(3):364--396, 2018.

\bibitem{MunteanuRoeckner2016}
I.~Munteanu and M.~R{\"o}{}ckner.
\newblock {The total variation flow perturbed by gradient linear multiplicative
  noise}.
\newblock {\em Infin. Dimens. Anal. Quantum. Probab. Relat. Top.},
  21(1):18500030 (28 pages), 2018.

\bibitem{PR07-2}
C.~Pr{\'e}v{\^o}t and M.~R{\"o}ckner.
\newblock {\em {A concise course on stochastic partial differential
  equations}}, volume 1905 of {\em Lecture Notes in Mathematics}.
\newblock Springer, Berlin, 2007.

\bibitem{RR2}
M.~M. Rao and Z.~D. Ren.
\newblock {\em {A}pplications of {O}rlicz spaces}.
\newblock Pure and Applied Mathematics. Marcel Dekker, Inc., New York--Basel,
  2002.

\bibitem{Rascanu:2017dp}
A.~R{\u{a}}{\c s}canu and E.~Rotenstein.
\newblock {Obstacle problems for parabolic SDEs with H{\"o}lder continuous
  diffusion: From weak to strong solutions}.
\newblock {\em J. Math. Anal. Appl.}, 450(1):647--669, 2017.

\bibitem{ReSi1}
M.~Reed and B.~Simon.
\newblock {\em {M}ethods of modern mathematical physics {I}. {F}unctional
  analysis}.
\newblock Academic Press, New York, revised and enlarged edition, 1980.

\bibitem{Rockafellar:1970ew}
R.~T. Rockafellar.
\newblock {On the maximality of sums of nonlinear monotone operators}.
\newblock {\em Trans. Amer. Math. Soc.}, 149(1):75--88, 1970.

\bibitem{SapountzoglouZimmermann2019}
N.~Sapountzoglou and A.~Zimmermann.
\newblock {Well-posedness of renormalized solutions for a stochastic
  $p$-Laplace equation with $L^1$ initial data}.
\newblock {\em Preprint}, pages 1--43, 2019.
\newblock \url{https://arxiv.org/abs/1908.11186}.

\bibitem{Shigekawa:2000io}
I.~Shigekawa.
\newblock {Semigroup domination on a Riemannian manifold with boundary}.
\newblock {\em Acta Appl. Math.}, 63(1-3):385--410, 2000.

\bibitem{Shigekawa:2006dp}
I.~Shigekawa.
\newblock {Defective intertwining property and generator domain}.
\newblock {\em J. Funct. Anal.}, 239(2):357--374, 2006.

\bibitem{Show}
R.~E. Showalter.
\newblock {\em {M}onotone operators in {B}anach space and nonlinear partial
  differential equations}.
\newblock Mathematical surveys and monographs, Amer. Math. Soc., 1997.

\bibitem{Toelle2018proceedings}
J.~M. T{\"o}lle.
\newblock {E}stimates for nonlinear stochastic partial differential equations
  with gradient noise via {D}irichlet forms.
\newblock In A.~Eberle, M.~Grothaus, W.~Hoh, M.~Kassmann, W.~Stannat, and
  G.~Trutnau, editors, {\em {Stochastic Partial Differential Equations and
  Related Fields, In Honor of Michael R\"ockner, SPDERF, Bielefeld, Germany,
  October 10--14, 2016}}, Springer Proceedings in Mathematics \& Statistics,
  229, pages 249--262. Springer International Publishing AG, Cham, 2018.

\bibitem{Turra2018}
M.~Turra.
\newblock {Existence and extinction in finite time for Stratonovich gradient
  noise porous media equations}.
\newblock {\em Evolution Equations and Control Theory}, 8(4):867--882, 2019.

\bibitem{ValletZimmermann2019}
G.~Vallet and A.~Zimmermann.
\newblock {Well-posedness for a pseudomonotone evolution problem with
  multiplicative noise}.
\newblock {\em J. Evol. Equ.}, 19(1):153--202, 2019.

\bibitem{vonRenesse:2005cg}
M.-K. von Renesse and K.-T. Sturm.
\newblock {Transport inequalities, gradient estimates, entropy, and Ricci
  curvature}.
\newblock {\em Comm. Pure Appl. Math.}, 58(7):923--940, 2005.

\bibitem{Wang}
F.-Y. Wang.
\newblock {\em {F}unctional inequalities, {M}arkov semigroups and spectral
  theory}.
\newblock Science Press, Beijing--New York, 2005.

\bibitem{Wang:2011fr}
F.-Y. Wang.
\newblock {Equivalent semigroup properties for the curvature-dimension
  condition}.
\newblock {\em Bull. Sci. Math.}, 135(6-7):803--815, 2011.

\bibitem{WangYan13}
F.-Y. Wang and L.~Yan.
\newblock Gradient estimate on convex domains and applications.
\newblock {\em Proc. Amer. Math. Soc.}, 141(3):1067--1081, 2013.

\end{thebibliography}
\end{document}